\documentclass[journal]{IEEEtran}
\usepackage{mathrsfs}
\usepackage{subfigure}
\usepackage[pdftex]{graphicx}
\usepackage{bm,amssymb,graphics,amsthm,amsopn,amstext,amsfonts,color,fullpage,hyperref}
\usepackage[cmex10]{amsmath}
\usepackage{longtable}
\usepackage{subfigure}
\usepackage{fullpage}
\usepackage{xcolor}

\usepackage{cite}
\usepackage[square,sort,comma,numbers]{natbib}
\usepackage{paralist}
\usepackage[]{authblk}
\usepackage{float}
\usepackage[bottom]{footmisc}
\usepackage{graphicx}
\usepackage{lineno}
\usepackage{natbib}
\usepackage{algorithm}
\usepackage[noend]{algpseudocode}
\usepackage{mathtools}
\DeclarePairedDelimiter\ceil{\lceil}{\rceil}
\DeclarePairedDelimiter\floor{\lfloor}{\rfloor}
\makeatletter
\def\BState{\State\hskip-\ALG@thistlm}
\makeatother
\usepackage{dblfloatfix}
\usepackage{array}
\newcolumntype{P}[1]{>{\centering\arraybackslash}p{#1}}
\newcolumntype{M}[1]{>{\centering\arraybackslash}m{#1}}

\newcommand{\epc}{\hspace{1pc}}

\newtheorem{theorem}{Theorem}
\newtheorem{proposition}[theorem]{Proposition}

\begin{document}

\title{Hierarchical Optimization of Charging Infrastructure Design and Facility Utilization}

\author{Amir~Mirheli~
	and~Leila~Hajibabai 
	\thanks{A. Mirheli is in Operations Research program with the Department of Industrial and Systems Engineering, North Carolina State University, NC 27695, USA. e-mail: amirhel@ncsu.edu .}
	\thanks{L. Hajibabai is with the Department of Industrial and Systems Engineering, North Carolina State University, NC 27695, USA. e-mail: lhajiba@ncsu.edu, to whom correspondence should be addressed.}
}

\markboth{}
{Mirheli and Hajibabai 
	: Hierarchical Optimization of Electric Vehicle Charging Infrastructure Design and Facility Logistics}

\maketitle

\begin{abstract}%
This study proposes a bi-level optimization program to represent the electric vehicle (EV) charging infrastructure design and utilization management problem with user-equilibrium (UE) decisions. The upper level aims to minimize total facility deployment costs and maximize the revenue generated from EV charging collections, while the lower level aims to minimize the EV users' travel times and charging expenses.
An iterative technique is implemented to solve the bi-level mixed-integer non-linear program that generates theoretical lower and upper bounds to the bi-level model and solves it to global optimality. A set of conditions are evaluated to show the convergence of the algorithm in a finite number of iterations. The numerical results, based on three demand levels, indicate that the proposed bi-level model can effectively determine the optimal charging facility location, physical capacity, and demand-responsive pricing scheme. The average charging price in medium demand level is increased by 38.21\% compared to the lower level demand due to the surge in charging needs and highly utilized charging stations.
\end{abstract}


\IEEEpeerreviewmaketitle

\section{Introduction}\label{s1}
\IEEEPARstart{T}{he} 
adoption of fully electric mobility technology, among alternative-fuel vehicles, can be considered as a promising solution to reduce environmental pollution and energy consumption caused by transportation systems. The trade-off between availability and cost of public charging resources can affect the large-scale adoption of the technology. Therefore, an effective resolution would be to locate sufficient charging facilities to serve the public charging demand in the transportation network. On the other hand, managing the charging load over time and space at the facilities will become very important: as the EV penetration rate grows, a large number of EVs will connect to the electricity grid to charge. Lack of proper management can lead to additional peak loads at certain times-of-day due to home charging attempts during evenings, end-of-day charging of EV fleets when returning to their businesses, among other factors. Consequently, large-scale EV penetration rates can lead to more frequent load peak occurrence at high magnitudes for long durations. Therefore, managing the charging facilities will be crucial to help distribute the charging load over the utilization periods.

This paper aims to determine optimal location and physical capacity for public charging infrastructure and applies a dynamic pricing strategy to manage the proposed charging facilities in an integrated infrastructure design and utilization management framework.
The proposed problem involves two entities, i.e., charging network operator (a.k.a. agency) and EV users, with hierarchical decisions, where the realized outcome of a decision made by the agency affects (is affected by) the users' decisions who seek to optimize their own outcomes. The conflicting objectives of the agency and users, sharing a competitive environment, can be represented by a bi-level optimization model that also captures the interactions between the two entities. 
The bi-level model structure can appropriately incorporate the two perspectives, where the charging agency (i.e., leader) defines its objectives in the upper level and passes the decisions to EV users (i.e., followers) in the lower level using a non-cooperative game theoretical strategy. 
The reaction of each entity to decisions taken by the other is captured with no binding agreement, where each player seeks to maximize its benefits.

The proposed bi-level optimization program represents the EV charging network design and utilization management under user-equilibrium (UE) decisions. The upper-level formulation minimizes total facility deployment costs and maximizes the revenue obtained from collecting EV charging fees. 
The lower-level model aims to minimize the EV users' travel and charging costs. 
Hence, the agency (upper-level) determines the charging locations, physical capacity, and charging pricing scheme, while EV users (lower-level) make charging decisions, based on appropriate charger availability, to meet their charging demand. 
Figure \ref{fig:introduction} presents the interactions between charging facility design (i.e., location and physical capacity) and EV users' charging choice (e.g., travel distance, waiting time at charging facilities). %
\vspace{-2mm}
\begin{figure}[h]
	\begin{center}
		\includegraphics[width=3.2 in]{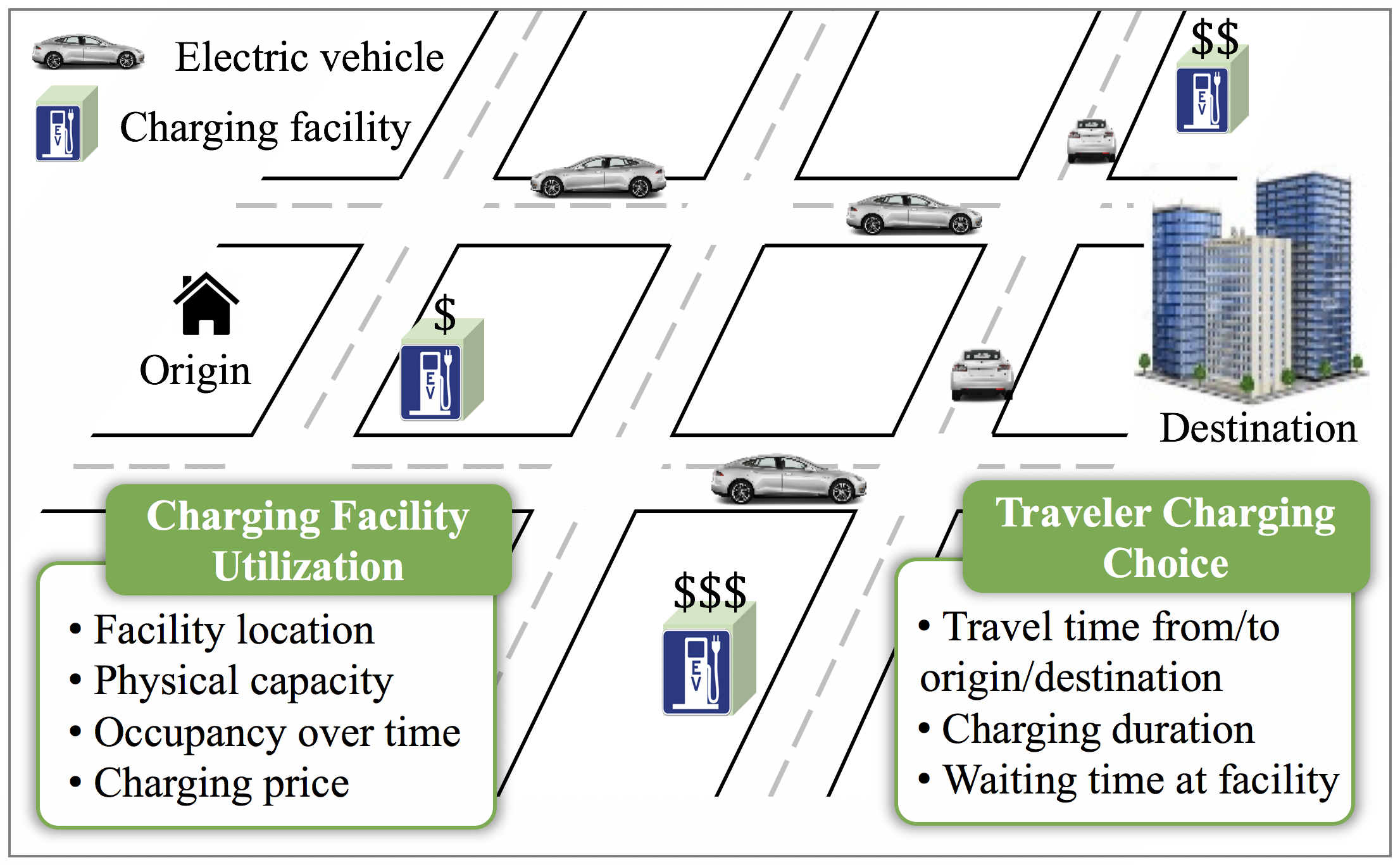}
		\vspace{-7mm}
		\caption{Interactions between EV charging facility utilization and EV travel flow.}
		\label{fig:introduction}
	\end{center}
\end{figure}
	\vspace{-4mm}
We utilize a technique proposed by Mitsos \cite{mitsos2010global} that solves the problem to global optimality by generating theoretical lower and upper bounds to the proposed bi-level problem. The proposed approach is applied to a hypothetical case study to evaluate the applicability of the proposed formulation and methodology to solve the EV charging network design and dynamic utilization management. The numerical results indicate that the algorithm can provide high quality solutions. A sensitivity analysis is conducted to study an alternative procedure in generating the cuts in the lower bounding procedure. Furthermore, a real-world dataset in Long Island, NY has been utilized to evaluate the computational efficiency of the algorithm.

The remainder of this paper is organized as follows. Section \ref{sec:Literature} introduces the existing literature in the charging station location domain and bi-level optimization methods. Section \ref{sec:modelformulation} presents the mathematical model formulation for dynamic charging facility location problem. Section \ref{sec:Algorithm} explains the proposed approach to solve the problem and Section \ref{sec:numericalresults} presents the numerical results. Finally, Section \ref{sec:conclusion} concludes the paper and discusses future research directions.
\vspace{-2mm}

\section{Literature Review}\label{sec:Literature}
We study relevant literature on (i) charging network design and operation in Section \ref{sec:Lit1} and (ii) bi-level optimization strategies in Section \ref{sec:Lit2}, as follows.
\vspace{-3mm}

\subsection{Charging Infrastructure Management}\label{sec:Lit1}
This section first summarizes a category of research, solely on the EV agency perspective, that focuses on the classical location models (e.g., daskin2011network,drezner2001facility) for re-charging EVs. 
The proposed models (i) capture the need for re-visiting EV facilities to replenish charge or fuel and (ii) introduce the range limitation concept to charging facility location problems. For example, Zheng and Peeta \cite{zheng2017routing} have researched optimal facility location problems with $p$-stop limit problem, where EVs should drive on paths with $p$ stops in a feasible range for intercity trips. Bahrami et al. \cite{bahrami2017complementarity} have also formulated a constrained shortest path problem under EV travel range limitation. 
In a related study, Zheng et al. \cite{zheng2017traffic} have used a bi-level structure, where the upper level determines the optimal location of charging facilities that minimizes total costs, including travel time and energy consumption, under range limitation constraints and the lower level determines EV traffic flow with regard to the feasible travel paths.
Besides, Zhang et al. \cite{zhang2017incorporating} have developed a flow re-fueling location model to find the optimal level-3 EV charging locations including the number of charging modules at each station. The problem is solved heuristically utilizing a forward strategy over multiple time periods considering dynamic demands and capacity constraints. 
Li et al. \cite{li2016design} have investigated EV sharing system design, including the fleet size and charging station deployment, that aims to minimize the cost of facility construction, transportation, and vehicle balancing under dynamic trip demands. This study employs a continuum approximation approach to find near-optimal solutions to the EV sharing system design problem.
Furthermore, Arslan and Karasan \cite {arslan2016benders} have explored charging infrastructure location design for EVs and plug-in hybrid vehicles using an arc-cover formulation that aims to maximize the EV travel mileage. This study employs a Benders decomposition technique to solve the problem. 
Similarly, Liu and Wang \cite{liu2017locating} have studied the deployment of different types of charging facilities (e.g., plug-in station, static wireless station, and dynamic wireless charging lane) to minimize public social costs under routing choice of battery electric vehicle (BEV) users. They have proposed a tri-level model as a black-box optimization program and solved it via a response surface approximation technique following a stochastic radial basis function method. 
Besides, Hof et al. \cite{hof2017solving} have studied a capacitated location-routing problem considering battery swap stations to minimize the total cost of facility construction and routing. An adaptive variable neighborhood search algorithm (originally designed for vehicle routing problems with intermediate stops) is used to solve the problem.
On the other hand, several studies have analyzed the impact of incentives on BEV adoption rates \citep{nie2016optimization} and EV routing decisions \citep{he2014network, xu2017network}. He et al. \cite{he2014network}, in particular, have studied the effect of travel range limitation on BEV users' path choice and equilibrium flow when charging facilities are sparse in the network.
Xu et al. \cite{xu2017joint} have developed a mixed logit model to study BEV users' choice of (i) charging mode (e.g., normal versus fast) and (ii) location of charging facilities (e.g., home versus public stations). The study identifies a set of factors that play a significant role in users' decisions, e.g., battery capacity and initial state-of-charge (SOC).

Another line of research on the EV infrastructure design focuses on the inter-relationships between specific destinations and demand variation of EV facilities. For instance, Momtazpour et al. \cite{momtazpour2012coordinated} have studied a coordinated clustering technique to identify appropriate charging locations based on several factors (e.g., residential ratios, electricity loads, and charging needs of each facility) to address the EV user demand from specific population groups (e.g., income levels and geographical locations).
By applying hierarchical clustering method, Ip et al. \cite{ip2010optimization} have determined the demand clusters of BEVs, representing road traffic information, to allocate the charging stations over an urbanized area under budget constraints. 
Simulation-based optimization methods have also been extensively used to identify appropriate deployment of charging facilities under certain criteria. 
For instance, Xi et al. \cite{xi2013simulation} have presented an integer programming model that aims to maximize the EV service level using slow-charging technologies for privately-owned EVs. 
They have analyzed the impacts of EV driving patterns and charging locations on traffic flows by simulating the inter-relationships between charging facility locations and associated service levels.
Moreover, Jung et al. \cite{jung2014stochastic} have developed a model in a bi-level simulation optimization form that integrates multiple EV charger allocation problems. The model aims to minimize the queuing delay under EV charger capacity constraints in the upper level, while minimizing the passenger waiting time in a taxi dispatch simulation subject to EV driving range and passenger detour length in the lower level problem.
Dong et al. \cite{dong2014charging} have utilized GPS-based travel survey data in the greater Seattle area and applied a genetic algorithm to find the location of charging stations while maximizing electric miles and minimizing the number of missed trips (i.e., when the required battery level to reach to a destination exceeds the remaining level).
Sweda and Klabjan \cite{sweda2011agent} have presented an agent-based decision support system, with underlying interactions between agents, that adopts the patterns observed in residential EV ownerships and driving activities to identify the best strategy for the deployment of charging facilities. 
Similarly, Asamer et al. \cite{asamer2016optimizing} have described another agent-based decision support system that captures the charging demand of EV taxis to deploy fast-charging facilities. This study uses a variant of maximum coverage location problem in an MILP to maximize the sum of covered taxi trip counts.

A category of literature has focused on EV user concerns on finding routes with charging facilities to guarantee that sufficient electricity is available for long trips. For instance, Adler et al. \cite{adler2016electric} have considered a feasible set of sub-trips in the network with maximum length that requires full battery to travel. The problem aims to minimize the travel length through a shortest-walk problem, where each EV is constrained to have one charging opportunity in each sub-trip. 
He et al. \cite{he2012optimal} have proposed a distributed scheduling problem that aims to minimize the EVs' charging costs. Each local controller bridges the connection between a group of EVs and the central controller to gather (i) predicted load of the day from the controller and (ii) real-time EV information from local charging stations.
Wang et al. \cite{wang2016path} have developed a distance-constrained traffic assignment model through an iterative linear approximation strategy, where the model is decomposed based on origin-destination pairs and activity sequences to solve the sub-problem in each iteration. They assume the activity sequences capture the EV users' behaviors (e.g., the tendency on charging location, travel range anxiety). 
Similarly, Wen et al. \cite{wen2016modeling} have studied the charging choices for EV drivers using stated preference data, obtained from U.S. battery electric vehicle owners. 
Three logit models are developed based on the information collected in the survey to indicate the EV users' charging choice under pre-defined situations. Besides, the interaction between dwell time, charging price at different locations, charging range, and the distance among charging facilities are derived and the effects of them on the probability of charging are tested. The results indicate an estimation for the willingness to pay the extra charging price for faster charging compared to level 1 charging option.
Besides, Yu and MacKenzie \cite{yu2016modeling} have studied the probability of having EV supply equipment at the end of the trip as well as the interaction of battery SOC and dwell time based on data from 125 pre-production Toyota Prius plug-in hybrid vehicles. 
Moreover, He et al. \cite{he2015deploying} have proposed a bi-level tour-based model, to minimize total costs including travel time and charging duration, considering public charging stations under drivers' risk-taking attributes given BEVs' SOC and limited driving range. An iterative approach is used to tackle the model complexity due to tour enumerations within a genetic algorithm framework that solves for the charging facility deployments. 
He et al. \cite{he2018optimal} have developed a bi-level optimization program to determine the optimal charging facility locations considering the driving range limitation and required charging duration by a path-based equilibrium traffic assignment. Rather than the amount of flow captured by charging facilities, the objective function models the maximum flow that can use charging facilities en-route. The proposed model structure and path-based traffic assignment reduce the computational efficiency. 
Liu and Song \cite{liu2018network} have proposed a non-linear complementarity model to determine the UE traffic assignment in a network with battery swapping stations for BEVs with the consideration of flow-dependent electricity consumption and driving-range limitation. The model is then formulated as a variational inequality, converted into a non-linear problem following the duality approach literature, and solved using a column generation algorithm (similar to \cite{hajibabai2019patrol,mehrabipour2019decomposition}). While the algorithm solves a large-scale UE problem, it assumes that the BEV demand is fixed.
\vspace{-3mm}

\subsection{Bi-level Optimization Strategies}\label{sec:Lit2}
Literature has shown various strategies to represent the inter-relationships between leaders and followers in bi-level optimization problems. A common approach is to find an equilibrium condition in the bi-level optimization program that satisfies the objectives of both levels. 
However, finding an optimal solution to a bi-level program has always been a challenge. A review study by Lu et al. \cite{lu2016multilevel} has revealed that the solution approaches for a linear bi-level decision-making problem can be classified into (i) implementation of vertex enumeration techniques \citep{bialas1984two}, (ii) application of Karush-Kuhn-Tucker (KKT) conditions and transformation of the bi-level optimization program into an equivalent single-level problem \citep{hajibabai2014joint,mirheli2020charging,bardaka2020reimagining}, and (iii) utilization of penalty function approaches \citep{anandalingam1990solution,white1993penalty}. 
For example, Mirheli and Hajibabai \cite{mirheli2019utilization} have formulated an integrated plan for parking utilization management, involving users and agency perspectives, into a bi-level optimization program. The problem is then converted into an equivalent single-level model and solved using a stochastic look-ahead technique, embedded in a Monte Carlo tree search algorithm, via a dynamic programming framework (similar to \cite{mirheli2018development,hajibabai2016dynamic}). 
Besides the existing approaches, literature shows that a number of meta-heuristic approaches have been implemented to solve bi-level optimization programs that offer higher computational efficiency, where their solution quality is not guaranteed.
For instance, Hejazi et al. \cite{hejazi2002linear} have developed a genetic algorithm (GA) to solve a linear bi-level optimization program, where 
the bi-level problem is converted into a single-level optimization problem and the GA is utilized to solve the single-level problem with complementary constraints. Each feasible chromosome represents an edge of the feasible region and simplifies the optimization model 
to considerably reduce the search space; similar to \cite{hajibabai2013integrated}. 
Besides, scenario-based techniques to solve bi-level programs have been proposed in the literature. For instance, the study by Xu and Wang \cite{xu2014exact} has presented an algorithm to solve a linear mixed-integer bi-level program under three simplified cases: finite optimal, infeasible, and unbounded problems.

Very limited studies have attempted to develop global optimization techniques to tackle the non-linearity (or non-convexity) in bi-level optimization programs (either in the lower- or upper-level formulation). 
Various simplifying assumptions have been made in previous studies to solve the bi-level problems under specific conditions (e.g., \cite{bard1988convex,edmunds1991algorithms}). 
For example, Bard \cite{bard1988convex} has assumed all functions (i.e., objective functions and the feasible region in both levels) are convex and proposed a hybrid technique to search for local optimal solutions in an inducible region. Then, a branching scheme has been developed to find the global optimum. 
Jan and Chern \cite{jan1994nonlinear} have studied a non-linear integer bi-level problem, where the upper level is formulated with no constraints and
the objective function and constraints of the lower level are defined as the summation of non-decreasing univariate terms.
Thirwani and Arora \cite{thirwani1997algorithm} have defined a bi-level program with a fractional linear objective function, linear constraints, and integer variables. 
Mersha and Dempe \cite{mersha2011direct} have replaced the bi-level program with an equivalent single-level problem under specific model structure, i.e., convex lower-level problem with strongly stable optimal solutions.
Approximation techniques are also proposed to solve bi-level problems that include non-convex functions. 
For instance, a study by Farvaresh and Sepehri \cite{farvaresh2011single} has approximated the link travel time through the definition of a piece-wise linear function to achieve a convex formulation. 
Besides, Al-Khayyal et al. \cite{al1992global} has replaced the complementary slackness condition by an equivalent system of convex and separable quadratic functions and developed
an integrated branch-and-bound and piece-wise linear approximation technique to find the global minimum.
%
%
Finally, a study by Mitsos et al. \cite{mitsos2008global} has proposed a bounding algorithm to find the global solution to a non-linear bi-level program that includes non-convex functions in both upper- and lower-level problems. A lower bounding problem is defined as a relaxed program including the constraints of lower- and upper-level models as well as a parametric upper bound to the optimal solution of the lower-level problem. Later, Mitsos \cite{mitsos2010global} has proposed an algorithm to find the global solution to non-linear mixed-integer bi-level programs. This study examines the impact of upper-level integer decision variables on the generation of parametric upper bounds to the lower-level problem. 
%
\vspace{-3mm}

\subsection{Summary}\label{sec:Summary}
This paper formulates the integrated network design and utilization management of EV charging facilities while accounting for charging agency and user perspectives. The problem is presented in a hierarchical optimization program that models the charging agency as a single leader in upper level and EV users as followers in the lower level using a non-cooperative game theoretical framework. 
While the literature is rich in the charging station location domain, the proposed integrated framework has not been addressed to the best of our knowledge.
The problem involves integer decision variables and non-linear terms in both upper- and lower-level formulations. Thus, an iterative technique is implemented to solve the problem to system-level optimality that generates theoretical lower and upper bounds to the proposed bi-level model \cite{mitsos2010global}.
\vspace{-3mm}

\section{Model Formulation}\label{sec:modelformulation}
This section introduces a bi-level mathematical program, where the upper level aims to minimize the one-time charger deployment expenses and maximize the charging network operator's revenue through effective utilization of suggested facilities. A demand-responsive dynamic pricing policy is incorporated into the upper-level model that is restricted by each facility's capacity constraints. Besides, a queuing theory driven model is embedded to gauge the level of service of each facility considering the average waiting time to find an available charging spot. On the other hand, the lower-level formulation aims to minimize the EV users' travel and charging expenses. The users' travel behavior is formulated as an equilibrium traffic assignment, where charger locations help identify the feasible EV paths and aggregated arc trips fulfill the origin-destination demands. 
We assume that EVs start their travels with a sufficient initial SOC to get to charging facilities and that EVs will leave charging facilities with enough SOC to get to their final destinations. Additionally, the utilization metric is defined by the charging duration at each facility. 
%
%

The planning horizon is defined by $\Gamma=\left\{0,1,\cdots,T-1\right\}$, where $T$ represents the number of discrete time steps at which we make the charging decisions.
Let $G(N,A)$ represent the transportation network with the set of all nodes $N$ and arcs $A$.
We define the set of inbound/outbound arcs to/from node $i\in N$ by $A_{i}^{-},\,A_{i}^{+} \subset A$. 
The decision variable $\eta_i \in \mathbb{Z}$ represents the physical capacity of a charging facility deployed at node $i$.
We denote the charging facility installation at node $i \in N$ by decision variable $y_i=\{0,1\}$. 
Besides, the charging pricing scheme of facility $i$ at each time $t$ is shown by decision variable $p_i^t$. 
The occupancy of charging facility at node $i\in N$ over time $t\in \Gamma$ is captured by state variable $f_i^t$. 
Hence, the available capacity $\hat{\eta}^t_i$ of facility at $i$ at time $t$ is calculated by subtracting the existing occupancy prior to time $t$ at node $i$ (i.e., $f^{t-1}_i$) from the physical capacity $\eta_i$, i.e., $\hat{\eta}^t_i=\eta_i-f^{t-1}_i$.
%

To ensure effective utilization of all charging facilities and avoid long queues in finding available charging spots (particularly, in high demand cases), a queuing theory concept is applied. 
We define function $\zeta$ to represent the probability of finding a vacant charger at a facility installed at node $i$ within a $\nu$ time window, i.e., $\zeta(\phi(f_{i}^{t},\eta_{i}) \le \nu)^t_i$, to measure the waiting time spent in finding available charging spots. Besides, the occupied facilities, e.g., in more popular areas, tend to be over-utilized more often. 
Hence, it is beneficial, for the arriving users, to capture the impact of usable capacity on each facility's expected service level. Therefore, function $\phi$ represents the relationship between physical capacity $\eta_i$ and occupancy $f^t_i$ of facility at node $i$ at time $t$. 
Similar to Xie et al. \cite{xie2018long}, we have utilized the Erlang C formula of queuing theory to characterize the waiting time since each facility with $\eta_i$ chargers can serve a queue of EV users with an expected service time $\theta$ and user arrival rate $\xi_{i}^{t}$ at time $t$. Hence, a non-linear relationship between $\eta_{i}$ and $f^t_i$ is introduced by
\begin{align}
	&\zeta(\phi(f_{i}^{t},\eta_{i}) \le \nu)^t_i = 1-\frac{(\xi_{i}^{t}\theta)^{\eta_{i}}}{\eta_{i}!}\Big(\frac{(\xi_{i}^{t}\theta)^{\eta_{i}}}{\eta_{i}!}+\nonumber\\
	&\epc \epc \epc \epc (1-\frac{\xi_{i}^{t}\theta}{\eta_{i}})\,\sum_{q=0}^{\eta_{i}-1}\frac{(\xi_{i}^{t}\theta)^q}{q!}\Big)^{-1}\, e^{-(\eta_{i}-\xi_{i}^{t}\theta){\nu}{\theta}^{-1}},\nonumber\\
	&\epc \epc \epc \epc \epc \epc \epc \epc \epc \epc \epc \epc \forall i \in N,\, t \in \Gamma.\,\label{eqref:waiting1}
\end{align}

We define the sets of travel origins and destinations by $O$ and $D$, respectively. 
Total elastic demand of EV users $\lambda_{od}^{t}$, follows an inverse demand function $H$ of equilibrium disutility $\sigma^{t}_{od}$ of users travelling on $od \in OD$ at time $t$, i.e.,
%
%
\begin{equation}
\lambda_{od}^{t} = H(\sigma_{od}^{t}) = g_{od}^{t} - b\,\sigma_{od}^{t},\,\forall od \in OD,\,t \in \Gamma,\,\label{eqref:ULDisutility1}
\end{equation}
where $H$ is defined as a linear function, $g_{od}^{t}$ is the intercept of EV demand curve, and $b$ is the demand elasticity coefficient. 
In other words, $\sigma^{t}_{od}$ represents the minimum cost imposed to EV users due to (1) charging price $p_i^t$ at charging facility $i \in N$ and time $t$ and (2) travel time from origin $o \in O$ to charging facility $i$ and beyond (i.e., final destination $d\in D$).
The EV traffic flow on arc $a=(i,j)\in A$ from $o\in O$ to $d\in D$, that visits charging facility installed at node $i \in N$ at time $t$, is represented by $x_{a}^{t,od,i}$.
The UE traffic flow of arc $a$ on $od \in OD$ at time $t$ is defined by the non-negative decision variable $z_{a}^{t,od}$, where $\sum_{a \in A_i^{+}} z_{a}^{t,od} = \sum_{a \in A_i^{-}} z_{a}^{t,od}$, unless node $i$ is either the origin or destination of user demand, i.e.,
\begin{align}
&\sum_{a \in A_i^{+}} z_{a}^{t,od} - \sum_{a \in A_i^{-}} z_{a}^{t,od} = \begin{cases}
		\lambda_{od}^{t}, & \forall i \in O \\[2pt]
		-\lambda_{od}^{t}, & \forall i \in D\\[2pt]
		0, & \text{O.W.}
		\end{cases}\nonumber\\
		&\epc \epc \epc \epc \epc \epc \epc \epc \epc \epc \forall od \in OD,\,t \in \Gamma,\,\label{eqref:LLCons2d}
\end{align}
which represents the conservation of EV user flow at node $i$.
Besides, the aggregated EV flow on arc $a$ at time $t$ is represented by $v_{a}^{t}$, as follows.
\begin{equation}
v_{a}^{t} = \sum_{od \in OD} z_{a}^{t,od},\,\forall a \in A,\, t \in \Gamma.\,\label{eq:LLCons2b}
\end{equation}
Travel time on arc $a\in A$ follows an increasing function $R_{a}^{t}$ of aggregated EV flow $v_a^t$, captured by Bureau of Public Roads' performance function \citep{Bureau1970}, represented by
\begin{equation}
R_{a}^{t}(v_{a}^{t}) = \hat{R}_{a}^{t}\big[1+w\left(\dfrac{v_{a}^{t}}{c_{a}}\right)^{q}\big], \label{eq:BPR}
\end{equation}
where $\hat{R}_{a}^{t}$ represents the free-flow travel time on arc $a$ at time $t$. In addition, $w$ and $q$ are the BPR function parameters, and $c_{a}$ is the traffic capacity of arc $a$. 

As EVs' driving range limit influences their route plan, we define feasible paths where EV users travel within their maximum driving range unless there is at least a charging facility en-route. We introduce variable $e_{a}^{t,od}\in\{0,1\}$ to identify arcs $a \in A$ on the feasible paths; i.e., $e_{a}^{t,od}$ is 1 if arc $a \in A$ is on a feasible path for EVs on $od \in OD$ at $t \in \Gamma$, or 0 otherwise.
We represent the range constraints (similar to \cite{zheng2017traffic}) at node $i$ on a path with $od \in OD$ with auxiliary variables $u_i^{od},\,u'^{od}_i>0$. Variable $u_i^{od}$ represents the maximum distance traveled from the last visited charging facility located at node $i \in N$ on a feasible path with $od \in OD$. In addition, $u'^{od}_i$ is defined as a dummy variable that is zero at facilities.
\begin{subequations}
	\begin{align}
	&
	z_{a}^{t,od} \le M\,e_{a}^{t,od},\,\forall a \in A,\, od \in OD,\, t \in \Gamma,\,\label{eqref:LLbinary}\\
	&
	u_{j}^{od} \ge u'^{od}_{i} + \delta_{a} -M\,(1-e_{a}^{t,od}),\nonumber\\
	&\epc \epc \epc \epc \epc \forall a=(i,j) \in A,\, od \in OD,\, t \in \Gamma,\, \label{eqref:rangeAnxiety1}\\
	&
	u_{i}^{od} \le \Delta,\,\forall i \in N,\,od \in OD,\, \label{eqref:rangeAnxiety2}\\
	&
	u'^{od}_{i} \ge u_{i}^{od} - M\,y_{i},\,\forall i \in N,\,od \in OD,\, \label{eqref:rangeAnxiety3}\\
	&
	u'^{od}_{i} \le u_{i}^{od} + M\,y_{i},\,\forall i \in N,\,od \in OD,\, \label{eqref:rangeAnxiety4}\\
	&
	u'^{od}_{i} \le M\,(1- y_{i}),\,\forall i \in N,\,od \in OD,\, \label{eqref:rangeAnxiety5}\\
	&
	u_{i}^{od} \ge 0,\,u'^{od}_{i} \ge 0,\,\forall i \in N,\,od \in OD,\, \label{eqref:rangeAnxiety6}
	\end{align}
\end{subequations}
where $\delta_a > 0$ represents the length of arc $a \in A$ and $\Delta$ denotes the driving range limit.
Constraints \eqref{eqref:LLbinary} ensure that EV flow selects arcs located on the feasible path. Constraints \eqref{eqref:rangeAnxiety1} update the auxiliary variable $u_j^{od}$ for all $j \in N$ and $od \in OD$ based on the traveled distance $d_a$ on the feasible path. Constraints \eqref{eqref:rangeAnxiety2} guarantees that EV users do not violate the driving range limit. Constraints \eqref{eqref:rangeAnxiety3} and \eqref{eqref:rangeAnxiety4} indicate that $u_i^{od} = u'^{od}_i$ when there is no charging facility at node $i$. 
Constraints \eqref{eqref:rangeAnxiety5} enforce $u'^{od}_i = 0$ when there is a charging facility at node $i$, where $M$ represents a large positive constant.
Finally, Constraints \eqref{eqref:rangeAnxiety6} represent that the value of auxiliary variable $u_i^{od},\, u'^{od}_i$ should be non-negative. 
The bi-level program is defined by
\begin{subequations}
	\begin{align}
		&\displaystyle{\operatornamewithlimits{\mbox{min}}_{\boldsymbol{y},\boldsymbol{\eta},\boldsymbol{p}}} \,\,\sum_{i \in N} \big(\eta_{i}\,C_{i}-\alpha\,\sum_{t \in \Gamma}\sum_{a\in A} \sum_{od \in OD} p_{i}^t\, x_{a}^{t,od,i} \big) , \label{eqref:UpperLevelObj}\\[2pt] 
		& \mbox{subject to} \,\,\,\,\,\,\, \eqref{eqref:waiting1}-\eqref{eqref:ULDisutility1}\,\,\, \mbox{and} \nonumber\\
		& l_i\,y_{i} \le p_{i}^{t} \le u_i\,y_{i},\,\forall i \in N,\,t \in \Gamma,\,\label{eqref:ULPriceRange1}\\[2pt]
		& f_{i}^{t} = f_{i}^{t-1} - \sum_{od \in OD}\sum_{a \in A_{i}^{+}} x_{a}^{t,od,i} + \sum_{od \in OD}\sum_{a \in A_{i}^{-}} x_{a}^{t,od,i}\,\nonumber\\
		&\epc \epc \epc \epc \epc \epc \epc \epc \epc \epc \epc \forall i \in N,\,t \in \Gamma,\,\label{eqref:ULOccupancyConsv1}\\[2pt]
		&  \sum_{od \in OD}\sum_{a \in A_{i}^{-}} x_{a}^{t,od,i} \le \hat{\eta}^t_{i},\,\forall i \in N,\,t \in \Gamma \setminus\{0\},\,\label{eqref:ULAllowedUsers1}\\[2pt]
		&  f_{i}^{t} \le \eta_{i},\,\forall i \in N,\,t \in \Gamma,\,\label{eqref:ULOccupancyLim1}\\[2pt]
		& \sum_{i \in N} \eta_{i}\,C_{i} \le B,\,\label{eqref:Budget1}\\[2pt]
		& \eta_{i} \le M\,y_{i},\,\forall i \in N,\,\label{eqref:ULCapacityLimit1}\\[2pt]
		& \eta_{i} \le \eta_{max},\,\forall i \in N,\,\label{eqref:StrategicCharger1}\\[2pt]
		& \zeta(\phi(f_{i}^{t},\eta_{i}) \le \nu)_i^t \ge \kappa,\,\forall i \in N,\, t \in \Gamma,\,\label{eqref:waiting}\\
		& \mbox{and} \epc \boldsymbol{x},\boldsymbol{z},\boldsymbol{v} \in \nonumber\\
		&\displaystyle{\operatornamewithlimits{\mbox{min}}_{\boldsymbol{x},\boldsymbol{z},\boldsymbol{v}}} \sum_{a \in A} \left(\int_{0}^{v_{a}^{t}} R_{a}^{t}(\omega) d\omega+ \gamma \sum_{i \in N}\sum_{od \in OD} p_{i}^{t}\,x_{a}^{t,od,i}\right) \label{eqref:LowerLevelObj} \\[2pt]
		& \mbox{subject to}\,\,\,\,\eqref{eqref:LLCons2d}-\eqref{eq:BPR},\,\eqref{eqref:LLbinary}-\eqref{eqref:rangeAnxiety6},\,\,\,\, \mbox{and} \nonumber \\[2pt]
		& x_{a}^{t,od,i} \le z_{a}^{t,od}, \forall  a=(i,j) \in A, od \in OD, t \in \Gamma,\,\label{eqref:LLCons2c}\\[2pt]
		& x_{a}^{t,od,i} \ge 0, z_{a}^{t,od} \ge 0, \forall a\in A, i\in N, od \in OD, t \in \Gamma, \label{eqref:LLNonNeg}\\[2pt]
		& v_{a}^{t} \ge 0,\, \forall a \in A,\,t \in \Gamma,\,\label{eqref:LLNonNeg1}	
	\end{align}
\end{subequations}
\noindent where $C_i$ is the unit charger installation cost at node $i \in N$ and $\alpha$ is a positive integer that represents the operation periods of each charging facility, e.g., number of operation days in a year.
The upper-level objective function \eqref{eqref:UpperLevelObj} aims to minimize the cost imposed by facility installation and maximize the revenue generated from charging collections. 
Constraints \eqref{eqref:ULPriceRange1} enforce a minimum and maximum charging price at time $t$ for an available charging facility at node $i$. 
Constraints \eqref{eqref:ULOccupancyConsv1} define the occupancy of facility at node $i$ at time $t$ considering its inbound and outbound EV flow.
Constraints \eqref{eqref:ULAllowedUsers1} ensure that EV traffic flow at time $t$ does not exceed the available capacity $\hat{\eta}^t_i$.
Constraints \eqref{eqref:ULOccupancyLim1} show that the occupancy of each facility at node $i$ at time $t$ cannot exceed its physical capacity.
Furthermore, charging facility deployment can be subject to budget constraints, as enforced in \eqref{eqref:Budget1}, where $B$ denotes the budget.
Constraints \eqref{eqref:ULCapacityLimit1} show that chargers can only be installed in an open charging facility at node $i$, where $M$ is a large positive value.
Constraints \eqref{eqref:StrategicCharger1} enforce a maximum physical capacity $\eta_{max}$ for each charging facility installed at node $i$.
Besides, constraints \eqref{eqref:waiting} enforce a lower bound $\kappa$ for the probability of finding an available charger in a facility installed at node $i$ at $t$.

The lower-level model \eqref{eqref:LowerLevelObj} aims to minimize EV users' travel times and charging expenses, where $\gamma$ denotes the monetary value of time.
Constraints \eqref{eqref:LLCons2c} ensure that EV charging demand at node $i$ does not exceed the flow through arc $a=(i,j)$ at time $t$. 
Finally, constraints \eqref{eqref:LLNonNeg}-\eqref{eqref:LLNonNeg1} represent the non-negativity of traffic flows.
\vspace{-2mm}

\section{Solution Technique}\label{sec:Algorithm}
The problem \eqref{eqref:waiting1}-\eqref{eq:BPR},\,\eqref{eqref:LLbinary}-\eqref{eqref:rangeAnxiety6},\,and \eqref{eqref:UpperLevelObj}-\eqref{eqref:LLNonNeg1} is a non-convex bi-level optimization program with mixed-integer decision variables and non-linear terms in both upper- and lower-level formulations.
We first introduce a linear function to capture the relationship between each facility's maximum occupancy and physical capacity with relatively high probability of finding an available charger within a reasonable time window, as follows.
%
%
%
\begin{equation}
	f_{i}^{t} \le \Omega_{i}\, \eta_{i},\,\forall i \in N,\, t \in \Gamma, \, \label{eqref:waitingApprox}
\end{equation}
where $\Omega_{i}$ represents the slope of the linear function. The proposed approximation provides an upper bound on the occupancy of facility located at node $i$.

The non-convexity of lower-level models introduces additional complexity to the bi-level problems. 
Hence, we apply a solution technique, developed by Mitsos \cite{mitsos2010global}, to generate theoretical lower and upper bounds and add cuts to the lower bounding procedure until the upper bound procedure generates an $\epsilon$-optimal point. We first introduce the underlying definitions and assumptions, and then describe the iterative exact algorithm to solve the proposed bi-level optimization program.
\vspace{-2mm}

\subsection{Definitions and Assumptions}\label{ExactMethod}
\noindent For notation simplicity, we let $F^{u}$ and $F^{l}$ represent the objective function of upper- and lower-level problems, respectively.
We first define the host sets, parametric optimal solution function, and candidate upper-level points as follows. The host set of all upper-level variables $(\boldsymbol{y},\boldsymbol{\eta},\boldsymbol{p})$ is defined as $U \equiv (\{0, 1\} \times \{0, 1, \dots, \eta_{max}\} \times [\operatornamewithlimits{min}_{i \in N}\, l_i, \dots, \operatornamewithlimits{max}_{i \in N}\, u_i])$. 
Similarly, we define the host set of lower-level variables $(\boldsymbol{x},\boldsymbol{z},\boldsymbol{v})$ as $L \equiv ([0, \dots, x_{max}] \times [0, \dots, z_{max}] \times [0, \dots, v_{max}])$. 
Let $F^{l*}(\boldsymbol{y},\boldsymbol{\eta},\boldsymbol{p})$ denote the parametric optimal value of the lower-level problem as a function of upper-level variables. We let $F^{l*}(\boldsymbol{y},\boldsymbol{\eta},\boldsymbol{p})=+\infty$ if no feasible solution is found to the lower-level problem.
We define $U^{\infty} \subset U$ that is applicable to both upper- and lower-level problems, where
\begin{align}
	&U^{\infty} = \big\{ (\boldsymbol{y},\boldsymbol{\eta},\boldsymbol{p}) \in U : \exists\, (\boldsymbol{x},\boldsymbol{z},\boldsymbol{v}) \in L :
	\eqref{eqref:ULDisutility1}-\eqref{eq:BPR},\,\nonumber\\
	&\,\,\, \eqref{eqref:LLbinary}-\eqref{eqref:rangeAnxiety6},\,\eqref{eqref:ULPriceRange1}-\eqref{eqref:StrategicCharger1},\,\eqref{eqref:LLCons2c}-\eqref{eqref:LLNonNeg1},\,\mbox{and}\, \eqref{eqref:waitingApprox}
\big\}.
\end{align}
Similarly, 
\begin{align}
&
U^{\psi}(\bar{F}^{u}) = 
\big\{ (\boldsymbol{y},\boldsymbol{\eta},\boldsymbol{p}) \in U : \exists (\boldsymbol{x},\boldsymbol{z},\boldsymbol{v}) \in L :\eqref{eqref:ULDisutility1}-\eqref{eq:BPR},\nonumber\\
&\epc \epc \epc \epc
 \eqref{eqref:LLbinary}-\eqref{eqref:rangeAnxiety6},\,\eqref{eqref:ULPriceRange1}-\eqref{eqref:StrategicCharger1}, \eqref{eqref:LLCons2c}-\eqref{eqref:LLNonNeg1},\nonumber \\
&\epc \epc \epc \epc \mbox{and}\, \eqref{eqref:waitingApprox}, 
\,F^{u}(\boldsymbol{y},\boldsymbol{\eta},\boldsymbol{p},\boldsymbol{x},\boldsymbol{z},\boldsymbol{v}) \le \bar{F}^{u} \big\},
\end{align}
where $\bar{F}^{u} \in \mathbb{R}$ represents an upper bound for the upper-level objective function $F^{u}$.

The following assumptions are made to ensure convergence of the algorithm in solving the proposed problem. We first assume that all variables should be defined with explicit bounds. The assumption will be satisfied by defining variables $(\boldsymbol{y},\boldsymbol{\eta},\boldsymbol{p},\boldsymbol{x},\boldsymbol{z},\boldsymbol{v})$. Second, all functions defined in the proposed bi-level program
\eqref{eqref:ULDisutility1}-\eqref{eq:BPR},\,\eqref{eqref:LLbinary}-\eqref{eqref:rangeAnxiety6},\,\eqref{eqref:UpperLevelObj}-\eqref{eqref:LLNonNeg1},\,\mbox{and}\, \eqref{eqref:waitingApprox} 
are assumed to be continuous for the given values of upper-level integer variables $(\boldsymbol{y},\boldsymbol{\eta}) \in \{0, 1\} \times \{0, 1, \dots, \eta_{max}\}$. By the continuity of the proposed constraints \eqref{eqref:ULDisutility1}-\eqref{eq:BPR},\,\eqref{eqref:LLbinary}-\eqref{eqref:rangeAnxiety6},\,\eqref{eqref:ULPriceRange1}-\eqref{eqref:StrategicCharger1},\,\eqref{eqref:LLCons2c}-\eqref{eqref:LLNonNeg1},\,\mbox{and}\, \eqref{eqref:waitingApprox} and the compact host sets $\{U,L\}$, we do not observe any non-continuity in the proposed constraints for known values of integer variables.
Third, for each vector of fixed upper-level values $(\bar{\boldsymbol{y}},\bar{\boldsymbol{\eta}},\bar{\boldsymbol{p}}) \in U^{\infty}$, there is a lower-level vector $(\tilde{\boldsymbol{x}}, \tilde{\boldsymbol{z}}, \tilde{\boldsymbol{v}}) \in L,\, \forall \varepsilon_{F1}^{l} \ge 0$, such that:
\begin{subequations}
	\begin{align}
	& 
	\eqref{eqref:LLCons2d}-\eqref{eq:BPR},\,\eqref{eqref:LLbinary}-\eqref{eqref:rangeAnxiety2},\,\eqref{eqref:rangeAnxiety6} ,\,\eqref{eqref:LLCons2c}-\eqref{eqref:LLNonNeg1},\,\mbox{and}\nonumber\\[2pt]
	&
	\tilde{u}_{i}^{od} - M\,\bar{y}_{i} - \tilde{u'}^{od}_{i} < 0,\,\forall i \in N,\,od \in OD,\, \label{eq:assumption1}\\[2pt]
	&
	\tilde{u'}^{od}_{i} - \tilde{u}_{i}^{od} + M\,\bar{y}_{i} < 0 ,\,\forall i \in N,\,od \in OD,\, \label{eq:assumption2}\\[2pt]
	&
	\tilde{u'}^{od}_{i} - M\,(1- \bar{y}_{i}) < 0,\,\forall i \in N,\,od \in OD,\, \label{eq:assumption3}\\[2pt]
	&
	\sum_{a \in A} \left(\int_{0}^{\tilde{v}_{a}^{t}} R_{a}^{t}(\omega)\,d\omega+ \gamma\,\sum_{i \in N}\sum_{od \in OD} \bar{p}_{i}^{t}\,\tilde{x}_{a}^{t,od,i}\right)\nonumber\\
	&\epc \epc \epc \epc \epc \epc \epc \le F^{l*}(\bar{\boldsymbol{y}}, \bar{\boldsymbol{\eta}}, \bar{\boldsymbol{p}}) + \varepsilon_{F1}^{l},\label{eq:assumption4}
	\end{align}
\end{subequations}
where $\varepsilon_{F1}^{l}$ represents the maximum interval for the parametric optimal solution of the lower-level program $F^{l*}(\bar{\boldsymbol{y}}, \bar{\boldsymbol{\eta}}, \bar{\boldsymbol{p}})$, to avoid infeasibility. 
Given values of upper-level decision variables $(\boldsymbol{y},\boldsymbol{\eta},\boldsymbol{p})$ help sketch the feasible network (satisfying \eqref{eqref:LLbinary}-\eqref{eqref:rangeAnxiety6} equations) that determines lower-level decision variables $(\boldsymbol{x},\boldsymbol{z},\boldsymbol{v})$. The lower-level problem defines an equilibrium traffic assignment for EVs driving on feasible paths and the solution always exists as there will be a feasible path for each $od \in OD$. Therefore, the lower bounding procedure finds feasible solutions based on the recently added cuts, declared by \eqref{eq:relaxedCons}. We will later see in Proposition \ref{Optimality} that the algorithm will lead to convergence given a feasible solution at each iteration.
%
\vspace{-2mm}

\subsection{The Global Optimization}\label{algorithm}
\noindent In this section, we generate theoretical lower and upper bounds to the original bi-level program \eqref{eqref:ULDisutility1}-\eqref{eq:BPR},\,\eqref{eqref:LLbinary}-\eqref{eqref:rangeAnxiety6},\,\eqref{eqref:UpperLevelObj}-\eqref{eqref:LLNonNeg1},\,and\, \eqref{eqref:waitingApprox} through an iterative algorithm to solve it to exact optimality. We first start the solution technique with a lower bounding procedure.
The proposed bi-level problem can be re-written based on the definitions introduced in Section \ref{ExactMethod} \citep{mitsos2010global}, as follows.
%
\begin{subequations}
	\label{eq:exact}
	\begin{align}
		&F^{u*}\,=\,\nonumber\\
		&\operatornamewithlimits{min}_{\boldsymbol{y},\boldsymbol{\eta},\boldsymbol{p},\boldsymbol{x},\boldsymbol{z},\boldsymbol{v}} \sum_{i \in N} \big(\eta_{i}\,C_{i}-\alpha \sum_{t \in \Gamma}\sum_{a\in A} \sum_{od \in OD} p_{i}^t x_{a}^{t,od,i}\big), \label{eqref:genral1}\\[2pt] 
		&\mbox{subject to} 
		\epc \eqref{eqref:ULDisutility1}-\eqref{eq:BPR},\,\eqref{eqref:LLbinary}-\eqref{eqref:rangeAnxiety6},\,\eqref{eqref:ULPriceRange1}-\eqref{eqref:StrategicCharger1},\nonumber\\
		&\epc \epc \epc \epc \,\, \eqref{eqref:LLCons2c}-\eqref{eqref:LLNonNeg1},\,\eqref{eqref:waitingApprox},\,\mbox{and}\nonumber\\
		& \epc \sum_{a \in A} \left(\int_{0}^{v_{a}^{t}} R_{a}^{t}(\omega)\,d\omega+ \gamma\,\sum_{i \in N}\sum_{od \in OD} p_{i}^{t}\,x_{a}^{t,od,i}\right)\,\nonumber\\
		&\epc \epc \epc \epc \epc \epc \epc \epc \epc \epc  \le\,F^{l*}(\boldsymbol{y},\boldsymbol{\eta},\boldsymbol{p}),\,\label{eqref:genral2}\\[2pt]
		& \epc \boldsymbol{y} \in \{0, 1\},\, \boldsymbol{\eta} \in \{0, 1, \dots, \eta_{max}\},\,\nonumber\\
		& \epc  \epc  \epc \boldsymbol{p} \in [\operatornamewithlimits{min}_{i \in N}\, l_i, \dots, \operatornamewithlimits{max}_{i \in N}\, u_i],\,	\boldsymbol{x} \in [0, \dots, x_{max}],\,\nonumber\\
		& \epc \epc  \epc \boldsymbol{z} \in [0, \dots, z_{max}],\,\boldsymbol{v} \in [0, \dots, v_{max}], \label{eqref:genral3}
	\end{align}
\end{subequations}
where $F^{u*}$ represents the optimal objective value of the bi-level problem.
We replace the right-hand side of \eqref{eqref:genral2} with the parametric upper bound of the lower-level problem and solve problem \eqref{eqref:ULDisutility1}-\eqref{eq:BPR},\,\eqref{eqref:LLbinary}-\eqref{eqref:rangeAnxiety6},\,\eqref{eqref:UpperLevelObj}-\eqref{eqref:LLNonNeg1},\,\eqref{eqref:waitingApprox},\,\mbox{and} \eqref{eqref:genral1}-\eqref{eqref:genral3} by adding cuts iteratively.
%
Each cut includes sets $U^{k}\,\subset\,U$ and points $(\boldsymbol{x}^k,\boldsymbol{z}^k,\boldsymbol{v}^k)\,\in\,L,\,\forall k\,\in\,K$, where $K$ is an index set that represents the collection of obtained cuts. 
In other words, we add a new cut to the problem in each iteration $k$ utilizing the defined subsets $U^{k-1}$ and points $(\boldsymbol{x}^{k-1},\boldsymbol{z}^{k-1},\boldsymbol{v}^{k-1})$ obtained in the previous iteration $k-1$.
%
Therefore, the problem can be relaxed, i.e., $\forall k\in K$, 
\begin{subequations}
	\label{eq:relaxed}
	\begin{align}
		& \mbox{LBD}\,=\,\nonumber\\
		&\operatornamewithlimits{min}_{\boldsymbol{y},\boldsymbol{\eta},\boldsymbol{p},\boldsymbol{x},\boldsymbol{z},\boldsymbol{v}} \,\,\sum_{i \in N} \big(\eta_{i}\,C_{i}-\alpha\,\sum_{t \in \Gamma}\sum_{a\in A} \sum_{od \in OD} p_{i}^t\, x_{a}^{t,od,i} \big),\,\label{eq:exactObj1}\\[2pt]
		&\mbox{subject to} \epc
		\eqref{eqref:ULDisutility1}-\eqref{eq:BPR},\,\eqref{eqref:LLbinary}-\eqref{eqref:rangeAnxiety6},\,\eqref{eqref:ULPriceRange1}-\eqref{eqref:StrategicCharger1},\nonumber\\
		&\epc \epc \epc \epc \,\,\eqref{eqref:LLCons2c}-\eqref{eqref:LLNonNeg1},\,\eqref{eqref:waitingApprox},\,\eqref{eqref:genral3},\,\mbox{and}\nonumber\\
		&
		\epc \boldsymbol{y},\boldsymbol{\eta},\boldsymbol{p} \in U^{k}\, \Rightarrow \, \nonumber\\
		&\epc\sum_{a \in A} \left(\int_{0}^{v_{a}^{t}} R_{a}^{t}(\omega)\,d\omega+ \gamma\,\sum_{i \in N}\sum_{od \in OD} p_{i}^{t}\,x_{a}^{t,od,i}\right)\,\nonumber\\
		&\epc \epc \epc \le\,F^{l}(\boldsymbol{y},\boldsymbol{\eta},\boldsymbol{p},\boldsymbol{x}^k,\boldsymbol{z}^k,\boldsymbol{v}^k),\,\forall k \in K,\,\label{eq:relaxedCons}
	\end{align}
\end{subequations}
where LBD represents the lower bound and is obtained from \eqref{eq:exactObj1}.
Iterative generation of sets $U^{k}$ and points $(\boldsymbol{x}^k,\boldsymbol{z}^k,\boldsymbol{v}^k)$ can ensure convergence of the lower bound through the following steps. 
First, the lower-level problem
\eqref{eqref:LLCons2d}-\eqref{eq:BPR},\,\eqref{eqref:LLbinary}-\eqref{eqref:rangeAnxiety6},\,\,and \eqref{eqref:LowerLevelObj}-\eqref{eqref:LLNonNeg1} is solved to global optimality given 
the value of upper-level decision variables $(\bar{\boldsymbol{y}},\bar{\boldsymbol{\eta}},\bar{\boldsymbol{p}})$, i.e., 
\begin{subequations}
	\label{eq:globalLL}
	\begin{align}
		& \bar{F}^{l*}\,=\,\nonumber\\
		&\operatornamewithlimits{min}_{\boldsymbol{x},\boldsymbol{z},\boldsymbol{v}} \sum_{a \in A} \left(\int_{0}^{v_{a}^{t}} R_{a}^{t}(\omega)\,d\omega+ \gamma\,\sum_{i \in N}\sum_{od \in OD} \bar{p}_{i}^{t}\,x_{a}^{t,od,i}\right),\,\label{eqref:lowerLevelglobal1}\\[2pt]
		&\mbox{subject to}\,\,
    	\eqref{eqref:LLCons2d}-\eqref{eq:BPR},\,\eqref{eqref:LLbinary}-\eqref{eqref:rangeAnxiety2},\,\eqref{eqref:rangeAnxiety6} ,
    	\eqref{eqref:LLCons2c}-\eqref{eqref:LLNonNeg1},\,\mbox{and}\nonumber\\[2pt]
		&
		u'^{od}_{i} \ge u_{i}^{od} - M\,\bar{y}_{i},\,\forall i \in N,\,od \in OD,\, \label{eqref:lowerLevelglobal2}\\[2pt]
		&
		u'^{od}_{i} \le u_{i}^{od} + M\,\bar{y}_{i},\,\forall i \in N,\,od \in OD,\, \label{eqref:lowerLevelglobal3}\\[2pt]
		&
		u'^{od}_{i} \le M\,(1- \bar{y}_{i}),\,\forall i \in N,\,od \in OD,\, \label{eqref:lowerLevelglobal4}\\[2pt]
		&
		\boldsymbol{x} \in [0, \dots, x_{max}],\,\boldsymbol{z} \in [0, \dots, z_{max}],\,\nonumber\\
		&\boldsymbol{v} \in [0, \dots, v_{max}], \label{eqref:lowerLevelglobal5}
	\end{align}
\end{subequations}
where $\bar{p}_{i}^{t}$ and $\bar{y}_{i}^{t}$ represent the fixed values of 
$p_{i}^{t}$ and $y_{i},\,\forall i \in N$, respectively.
The second step is to find a triple of lower-level variables $(\boldsymbol{x}^k,\boldsymbol{z}^k,\boldsymbol{v}^k) \in L$ 
given the optimal value of upper-level variables $(\bar{\boldsymbol{y}},\bar{\boldsymbol{\eta}},\bar{\boldsymbol{p}})$ $,\,\forall\,\varepsilon_{F2}^{l} \ge 0$, as follows.
\begin{subequations}
	\label{eq:secondStep}
	\begin{align}
		& \Upsilon^{*}\,=\,\operatornamewithlimits{min}_{\boldsymbol{x},\boldsymbol{z},\boldsymbol{v}} \Upsilon\,\label{eq:secondStepObj}\\
		&\mbox{subject to} \epc 
		\eqref{eqref:LLCons2d}-\eqref{eq:BPR},\eqref{eqref:LLbinary}-\eqref{eqref:rangeAnxiety2},\eqref{eqref:rangeAnxiety6} ,\nonumber\\
		&\epc \epc \epc \epc \,\,
		\eqref{eqref:LLCons2c}-\eqref{eqref:LLNonNeg1},\,\eqref{eqref:lowerLevelglobal5},\,\mbox{and}\nonumber\\
		&
		\sum_{a \in A} \left(\int_{0}^{v_{a}^{t}} R_{a}^{t}(\omega)\,d\omega+ \gamma\,\sum_{i \in N}\sum_{od \in OD} \bar{p}_{i}^{t}\,x_{a}^{t,od,i}\right) \nonumber\\
		& \epc \epc \epc \epc \epc \epc \epc \epc \epc \epc \epc \le \bar{F}^{l*} + \varepsilon_{F2}^{l},\,\label{eq:secondStepCons1}\\
		&
		u_{i}^{od} - M\,\bar{y}_{i} - u'^{od}_{i} \le \Upsilon,\,\forall i \in N,\,od \in OD,\, \label{eq:secondStepCons2}\\
		&
		u'^{od}_{i} - u_{i}^{od} - M\,\bar{y}_{i} \le \Upsilon ,\,\forall i \in N,\,od \in OD,\, \label{eq:secondStepCons3}\\
		&
		u'^{od}_{i} - M\,(1- \bar{y}_{i}) \le \Upsilon,\,\forall i \in N,\,od \in OD,\, \label{eq:secondStepCons4}
	\end{align}
\end{subequations}
where $\Upsilon$ represents an auxiliary variable. 
The final step is to find subsets $U^k$ using the bounds of upper-level decision variables $(\boldsymbol{y},\boldsymbol{\eta},\boldsymbol{p})$ that satisfy the constraints of the lower-level problem \eqref{eqref:LLCons2d}-\eqref{eq:BPR},\,\eqref{eqref:LLbinary}-\eqref{eqref:rangeAnxiety6},\, and \eqref{eqref:LLCons2c}-\eqref{eqref:LLNonNeg1}.
Therefore, we implement a sub-routine applied by Oluwole et al. \cite{oluwole2006rigorous} and Mitsos \cite{mitsos2010global}.
For a given point $(\bar{\boldsymbol{y}},\bar{\boldsymbol{\eta}},\bar{\boldsymbol{p}})$, a point $(\bar{\boldsymbol{x}}^k, \bar{\boldsymbol{z}}^k, \bar{\boldsymbol{v}}^k)$ and $U$, 
we calculate the bounds of box $U^{k} = \big[\bar{\boldsymbol{y}}^{k,l},\,\bar{\boldsymbol{y}}^{k,u}\big] \times \big[\bar{\boldsymbol{\eta}}^{k,l},\,\bar{\boldsymbol{\eta}}^{k,u}\big] \times \big[\bar{\boldsymbol{p}}^{k,l},\,\bar{\boldsymbol{p}}^{k,u}\big]$ and construct smaller boxes through interations to finally obtain a conservative estimate of lower-level constraints, as described in Algorithm \ref{subroutine}. 
%
%
%
For simplicity, the upper-level decision variables $(\boldsymbol{y},\boldsymbol{\eta},\boldsymbol{p})$ are shown by $\hat{\boldsymbol{\omega}}$ in the following subroutine.  
\vspace{-4pt}
\begin{algorithm}[H]
	\caption{Subroutine}\label{subroutine}
	\begin{algorithmic}[1]
		\State Set $\mu^{0}=1$
		\State \textit{Loop} $\tau = 0,\dots, T'$
		\Statex \hspace{1mm} (a) \textit{for} $m=1,\dots, |\Gamma|+|N|$
		\Statex \hspace{5mm} \textit{if} $\bar{\omega}_{m}-\frac{\mu^{\tau}}{2}(\hat{\omega}_{m}^{u}-\hat{\omega}_{m}^{l}) < \hat{\omega}_{m}^{l}$, \textit{then} 
		\Statex \hspace{10mm} Set $\hat{\omega}_{m}^{k,l} = \hat{\omega}_{m}^{l},$
		$\hat{\omega}_{m}^{k,u} = \hat{\omega}_{m}^{l}+d(\hat{\omega}_{m}^{u}-\hat{\omega}_{m}^{l})$.
		\Statex \hspace{5mm} \textit{else if} $\bar{\omega}_{m}+\frac{\mu^{\tau}}{2}(\hat{\omega}_{m}^{u}-\hat{\omega}_{m}^{l}) > \hat{\omega}_{m}^{u}$,  \textit{then}
		\Statex \hspace{10mm} Set $\hat{\omega}_{m}^{k,l} = \hat{\omega}_{m}^{u}-\mu^{\tau}(\hat{\omega}_{m}^{u}-\hat{\omega}_{m}^{l}),$
		$\hat{\omega}_{m}^{k,u} = \hat{\omega}_{m}^{u}$.
		\Statex \hspace{5mm} \textit{else}
		\Statex \hspace{10mm} Set $\hat{\omega}_{m}^{k,l} = \hat{\omega}_{m}^{u}-\frac{\mu^{\tau}}{2}(\hat{\omega}_{m}^{u}-\hat{\omega}_{m}^{l}),$
		\Statex \hspace{16mm}
		$\hat{\omega}_{m}^{k,u} = \hat{\omega}_{m}^{u}+\frac{\mu^{\tau}}{2}(\omega_{m}^{u}-\hat{\omega}_{m}^{l})$.
		\Statex \hspace{1mm} (b) \textit{for} $m=1,\dots, |N|$ 
		\Statex \hspace{5mm} Set $y_m^{k,l} = \ceil{ y_m^{k,l}}, \eta_m^{k,l} = \ceil{\eta_m^{k,l}},$
		\Statex \hspace{11mm}
		$y_m^{k,u} = \floor{ y_m^{k,u}}, \eta_m^{k,u} = \floor{\eta_m^{k,u}}$.
		\Statex \hspace{1mm} (c) Check lower-level constraints
		\eqref{eqref:LLCons2d}-\eqref{eq:BPR},\,\eqref{eqref:LLbinary}-\eqref{eqref:rangeAnxiety6},
		\Statex \hspace{6mm} and \eqref{eqref:LLCons2c}-\eqref{eqref:LLNonNeg1} 
		with given point $\bar{\omega}^{k}$ on $U^{k}$
		\Statex \hspace{10mm} \textit{if} the constraints are satisfied, 
		\textit{then}
		\Statex \hspace{10mm} Terminate.
		\Statex \hspace{10mm} \textit{else} Set $\mu^{\tau}=\mu^{\tau+1}$.
		\Statex End.
	\end{algorithmic}
\end{algorithm}

\vspace{-3mm}

We now resume the exact solution technique through an upper bounding procedure. The upper bound of the proposed bi-level program is obtained by solving the following problem
that includes the constraints of lower-level problem given the values of upper-level decision variables, i.e., points $(\bar{\boldsymbol{y}},\bar{\boldsymbol{\eta}},\bar{\boldsymbol{p}})$ found in the lower bounding procedure. 
%
%
%
\vspace{-1mm}
\begin{subequations}
	\label{eq:upperBound}
	\begin{align}
		&\mbox{UBD}=\operatornamewithlimits{min}_{\boldsymbol{x},\boldsymbol{z},\boldsymbol{v}} \sum_{i \in N} \big(\bar{\eta}_{i} C_{i}-\alpha\,\sum_{t \in \Gamma}\sum_{a\in A} \sum_{od \in OD} \bar{p}_{i}^t x_{a}^{t,od,i} \big), \label{eq:upperBoundObj}\\
		&\mbox{subject to}\epc \eqref{eqref:LLCons2d}-\eqref{eq:BPR},\,\eqref{eqref:LLbinary}-\eqref{eqref:rangeAnxiety2},\,\eqref{eqref:rangeAnxiety6},\nonumber\\
		&\epc \epc \epc \epc \,\, \eqref{eqref:LLCons2c}-\eqref{eqref:LLNonNeg1},\,\eqref{eqref:lowerLevelglobal2}-\eqref{eqref:lowerLevelglobal5},\,\mbox{and}\nonumber\\
		&\bar{f}_{i}^{t} = \bar{f}_{i}^{t-1} - \sum_{od \in OD}\sum_{a \in A_{i}^{+}} x_{a}^{t,od,i}\nonumber\\
		&\epc \epc \epc \,  + \sum_{od \in OD}\sum_{a \in A_{i}^{-}} x_{a}^{t,od,i},\,\forall i \in N,\,t \in \Gamma,\,\label{eq:upperBound1}\\
		&\sum_{od \in OD}\sum_{a \in A_{i}^{-}} x_{a}^{t,od,i} \le \bar{\eta}_{i} - \bar{f}_{i}^{t-1},\forall i \in N,t \in \Gamma\setminus\{0\},\label{eq:upperBound2}\\
		&\lambda_{od}^{t} = H(\bar{\sigma}_{od}^{t}) = g_{od}^{t} - b \bar{\sigma}_{od}^{t},\forall od \in OD, t \in \Gamma,\label{eq:upperBound3}\\
		& 
		\sum_{a \in A} \left(\int_{0}^{v_{a}^{t}} R_{a}^{t}(\omega)\,d\omega+ \gamma\,\sum_{i \in N}\sum_{od \in OD} \bar{p}_{i}^{t}\,x_{a}^{t,od,i}\right)\nonumber\\
		&\epc \epc \epc \epc \epc \epc \epc \epc \epc \epc \epc \le \bar{F}^{l*}+\varepsilon_{F}^{l},\,\label{eq:upperBound7}\\
		&
		\mbox{LBD} \le \sum_{i \in N} \big(\bar{\eta}_{i} C_{i}-\alpha\,\sum_{t \in \Gamma}\sum_{a\in A} \sum_{od \in OD} \bar{p}_{i}^t x_{a}^{t,od,i} \big).\label{eq:upperBound8}
	\end{align}
\end{subequations}
%
where $\varepsilon_{F}^{l}$ represents the maximum violation of the lower-level objective value $\bar{F}^{l*}$. 
The proof of convergence of the proposed exact algorithm is described below, based on Mitsos \cite{mitsos2010global}.
\begin{proposition} \label{Optimality}
	Based on Theorem 1 in Mitsos \cite{mitsos2010global}, the proposed exact algorithm terminates finitely if the optimality gap of the mixed-integer non-linear problem (MINLP), i.e., $\tilde{\varepsilon}$ and $\varepsilon_{F2}^{l}$, satisfies
 	\begin{subequations}
		\label{eq:proofCond}
 		\begin{align}
 			&
			0 < \tilde{\varepsilon} \le \mbox{min}{\{\varepsilon_{F}^{u}/2,\,\bar{\varepsilon}_{F}^{u},\,\varepsilon_{F1}^{l}\}}\,\,\,\mbox{and} 
 			\,\,\, 0 < \varepsilon_{F2}^{l} < \varepsilon_{F}^{l} - \tilde{\varepsilon}.\nonumber
 		\end{align}
 	\end{subequations}	
 	%
	
	
\end{proposition}
\begin{proof}
See Appendix \ref{sec:appProof}.
\end{proof}

\section{NUMERICAL EXPERIMENTS}\label{sec:numericalresults}
\noindent 
We have coded our model \eqref{eqref:ULDisutility1}-\eqref{eq:BPR},\,\eqref{eqref:LLbinary}-\eqref{eqref:rangeAnxiety6},\,\eqref{eqref:UpperLevelObj}-\eqref{eqref:LLNonNeg1}, and \eqref{eqref:waitingApprox}, and the solution technique described in Section \ref{sec:Algorithm} in Java, utilizing a commercial solver LINDO \citep{lin2009global}. 
Parameters $w$ and $q$ in BPR function \eqref{eq:BPR} are set to $0.15$ and $4$, receptively.
A Poisson distribution is applied to generate the initial demand patterns for early AM, AM peak, mid-day, PM peak, and evening time-of-days in a business day.
\vspace{-2mm}

\subsection{Hypothetical Dataset}\label{subsec:hypo}
\noindent This section first summarizes the assumptions and then presents the numerical results and sensitivity analyses based on a hypothetical network dataset including 18 nodes and 58 links, as shown in Figure \ref{fig:hypoNetwork}.
%
We assume that 30\% of EV travel demand originates from node 1 and completes in node 8, while the remainder goes to node 11.
We have assumed a planning horizon from 8 AM to 5:30 PM with 30 $min$ time periods. EV users are assumed to get charged in one time period.
The average vehicle arrivals over time-of-days are assumed to be 10, 20, 15, 10, and 20 for early AM, AM peak, mid-day, PM peak, and evening, respectively, for a medium demand level.
We assume that the low demand level is half of the medium demand.
Assuming the same demand distribution on different days, we set the objective function parameter $\alpha$ to 365, representing the number of days in a year. 
We assume a maximum driving range limit of 15 $miles$ for EVs. All network nodes except the ones located on the outer border (e.g., nodes 1 to 12) are assumed as candidate locations for the deployment of EV facilities. Furthermore, it is assumed that at most three chargers can be deployed in a selected charging facility.
The upper and lower bounds of charging prices per time step are assumed to be $u_i=\$15.0$ and $l_i=\$0.1, \forall i \in N$.
We have run the code on a desktop with quad-core 3.6 GHz CPU and 16 GB of memory.
\vspace{-3mm}
\begin{figure}[h]
	\begin{center}
		\includegraphics[height=1.90 in]{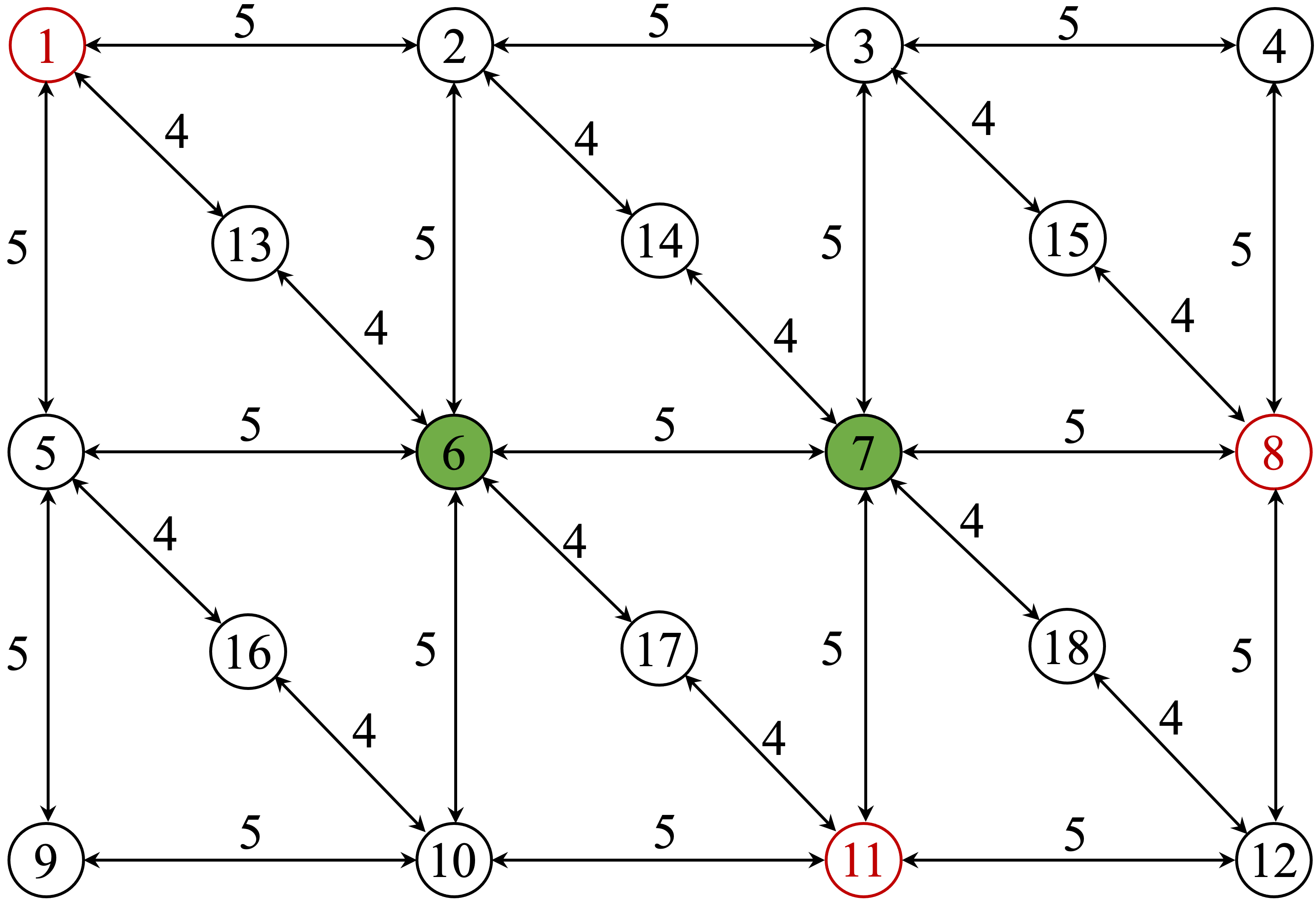}
		\vspace{-2mm}
		\caption{Hypothetical network. \hspace{5in}}
		\label{fig:hypoNetwork}
	\end{center}
\end{figure}
\vspace{-3mm}
\subsubsection{Results}\label{subsec:results}
EV travels are originated from node 1 and completed at nodes 8 and 11 in a day from 8 AM to 5:30 PM. Considering a medium demand level, nodes 6 and 7 are selected as optimal locations to deploy EV charging facilities, as shown in Figure \ref{fig:hypoNetwork}. The physical capacity of facilities installed at nodes 6 and 7 is 4 and 2 chargers, respectively.
To address the EV users' range anxiety concern, Figure \ref{fig:maxdist} shows the maximum distance traveled from the last visited EV charging facility for the same demand level. As indicated, the travel range limit of 15 $miles$ is not violated by EV users, which confirms the appropriate location of selected charging facilities in the network.
\vspace{-3mm}
\begin{figure}[h]
	\begin{center}
		\includegraphics[height=1.75 in]{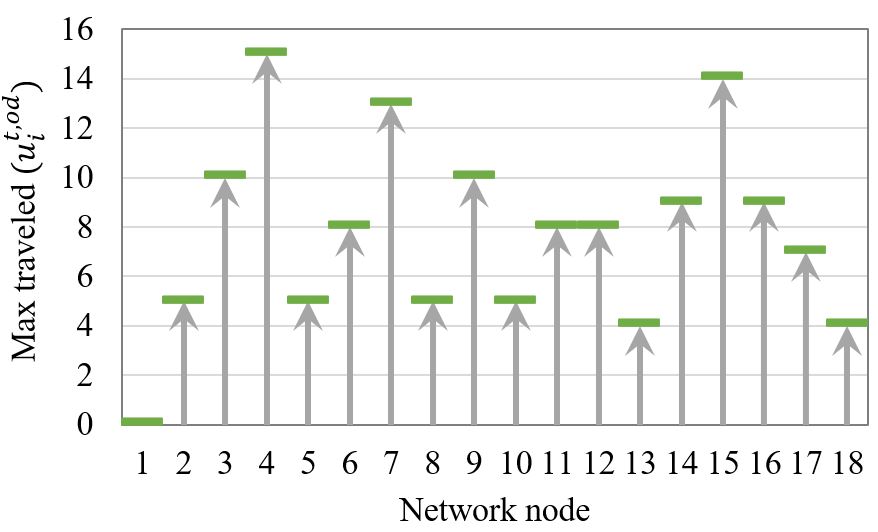}
		\vspace{-3mm}
		\caption{Max distance traveled from last visited node in med. demand.}
		\label{fig:maxdist}
	\end{center}
\end{figure}
\vspace{-5mm}
Table \ref{table:ChargingPrice} presents the charging pricing scheme in the selected EV facilities. Since charging demand from node 1 to node 11 is high, charging prices in facility 6 are higher than facility 7 to properly respond to the user arrivals.
Besides, the charging prices increase during the AM and PM peak hours due to significantly higher demand that increases the waiting time of users to find available chargers.
When user arrivals decrease, so does the charging prices, because more chargers are vacant.
\vspace{-4mm}
\begin{table}[H]
	\caption{Charging price (\$) at selected facilities over the planning horizon. \hspace{3in}}
		\vspace{-2mm}
	\begin{center}
		\small
		\begin{tabular}{c c c c c c c c c c c}
			\hline
			\hline
			\textbf{Time step} & 08:00 & 08:30 & 09:00 & 09:30  \\[0.5ex]
			\hline			
			Facility at node 6 & 0.15 & 0.45 & 0.87 & 3.34 \\
			Facility at node 7 & 0.15 & 0.36 & 0.71 & 2.83 \\[5pt]
			\textbf{Time step (cont.)} & 10:00 & 10:30 & 11:00 & 11:30\\
			\hline
			Facility at node 6 &  5.13 & 5.27 & 6.21 & 5.08\\
			Facility at node 7 & 4.32 & 4.92 & 5.52 & 4.10\\[5pt]
			\textbf{Time step (cont.)} & 12:00 & 12:30 & 13:00 & 13:30\\
			\hline
			Facility at node 6 & 4.62 & 4.79 & 4.38 & 2.92 \\
			Facility at node 7 & 4.44 & 3.66  & 3.80 & 2.84 \\[5pt]
			\textbf{Time step (cont.)}  & 14:00 & 14:30 & 15:00 & 15:30\\
			\hline
			Facility at node 6 & 2.47 & 3.40 & 5.12 & 5.69 \\
			Facility at node 7 & 2.26 & 2.90 & 4.00 & 4.26 \\[5pt]
		    \textbf{Time step (cont.)}  & 16:00 & 16:30 & 17:00 & 17:30\\
			\hline
			Facility at node 6 & 7.64 & 9.04 & 10.79 & 11.31\\
			Facility at node 7 & 6.18 & 7.72 & 8.83 & 9.05\\
			\hline
		\end{tabular}
		\label{table:ChargingPrice}
	\end{center}
\end{table}
	\vspace{-4mm}
%

%

Table \ref{table:Average} presents the average and standard deviation of charging prices over a day from 8 AM to 5:30 PM for low and medium demand levels.
The standard deviation indicates the impact of user arrivals on the pricing scheme over time. 
As indicated, both values increase by 23.36\% and 38.69\% in medium compared to low demand level.
Numerical results show that the number of chargers is affected by various demand distributions. For instance, the number of chargers shall be increased from one to two at charging facility 7 when the demand level switches from low to medium. Therefore, it is expected that the agency installs at least two chargers at the designated node to ensure higher demand satisfaction in the future.
Similarly, the agency deploys a facility at node 6 to serve users in all demand levels, while the number of chargers depends on the demand intensity.  
%
%
\vspace{-4mm}
\begin{table}[H]
	\caption{Average and standard deviation of charging prices at selected facilities for low and medium demand levels.}
	\begin{center}
		\small
		\begin{tabular}{M{1.0cm} M{1.7cm} M{1.7cm} M{1.7cm}}
			\hline
			\hline
			 & Low demand & Med demand & \% Diff \\[0.5ex]
			\hline			
			avg (\$) & 3.68 & 4.54 & 23.36 \\
			std (\$) & 1.99 & 2.76 & 38.69 \\
			\hline
		\end{tabular}
		\label{table:Average}
	\end{center}
\end{table}
	\vspace{-4mm}

Figure \ref{fig:bounds} shows the convergence of upper and lower bounds with a gap of 4.58\% in the bi-level optimization program. The CPU time for the algorithm is 21.4 $hr$ in the medium demand level. As indicated in Figure \ref{fig:bounds}, adding cuts in the lower bounding procedure improves the solution during the initial iterations, while its improvement rate decreases as the tightening procedure of the feasible region in lower bounding proceeds. 
\vspace{-2mm}
\begin{figure}[H]
	\begin{center}
		\includegraphics[height=1.30 in]{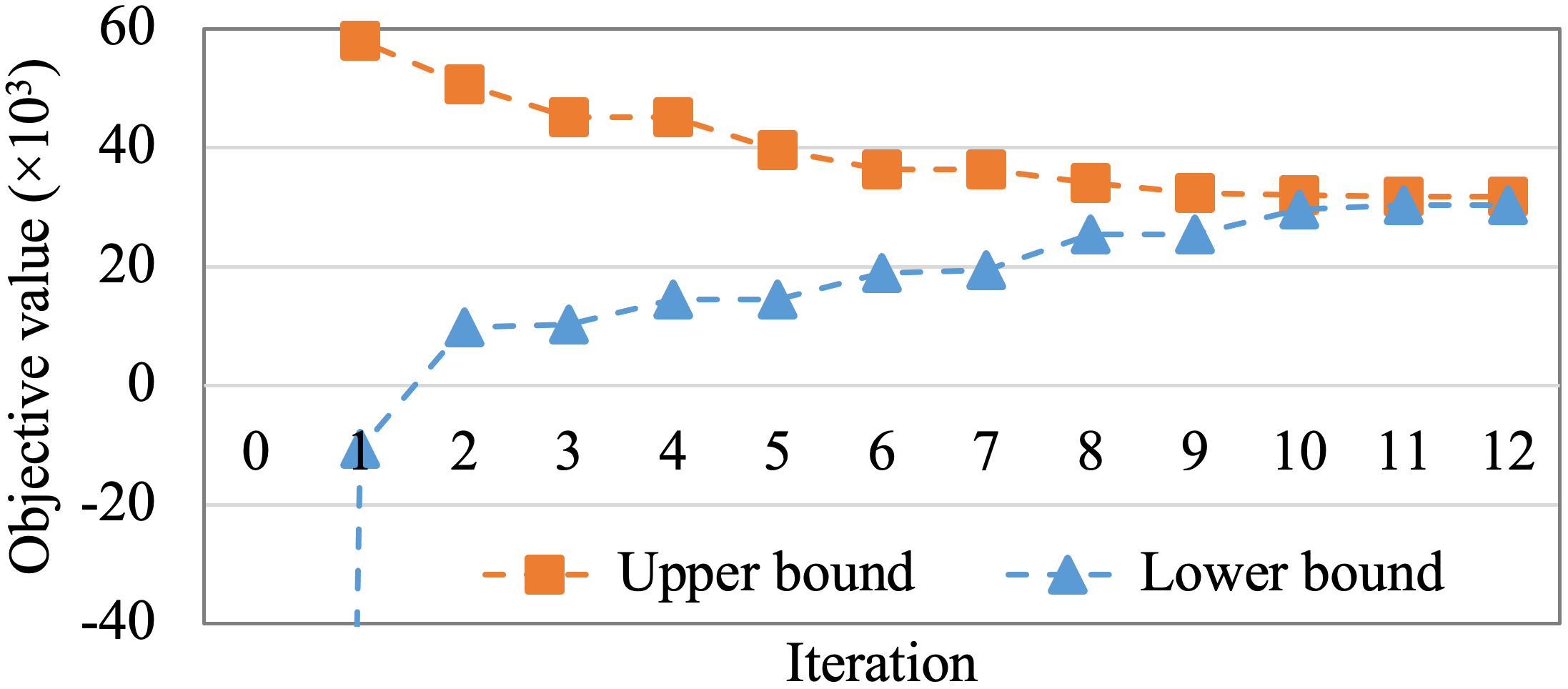}
		\vspace{-2mm}
		\caption{Convergence of upper bound and lower bound (\$).}
		\label{fig:bounds}
	\end{center}
\end{figure}
\vspace{-10pt}
\subsubsection{Sensitivity Analysis}\label{subsec:sensitivity}
\noindent We perform an analysis to show the sensitivity of the solutions to parameter $\mu$ in Algorithm \ref{subroutine}. Hence, an alternative updating procedure is implemented to obtain $\mu$, as shown in Figure \ref{fig:SBU_result2}. 
In the one procedure, the value of $\mu$ is divided by 2 at each iteration, while in the alternative case, the iteration number affects the value of $\mu$ as $\mu=1/\tau$ (see Figure \ref{fig:SBU_result2}). Tightening the bounds of upper-level decision variables $(\boldsymbol{y},\boldsymbol{\eta},\boldsymbol{p})$ by a factor of 2 in successive iterations leads to a convergence after 12 iterations with a gap of 4.58$\%$ in 21.4 $hr$. However, the alternative updating procedure based on $\mu=1/\tau$ increases the number of iterations by 58.33$\%$, and therefore the CPU time increases by 8.87$\%$ (i.e., 23.3 $hr$) as the alternative procedure offers a more relaxed feasible region in the lower bounding procedure.
%
\vspace{-6mm}
\begin{figure}[H]
	\begin{center}
		\includegraphics[height=1.65 in]{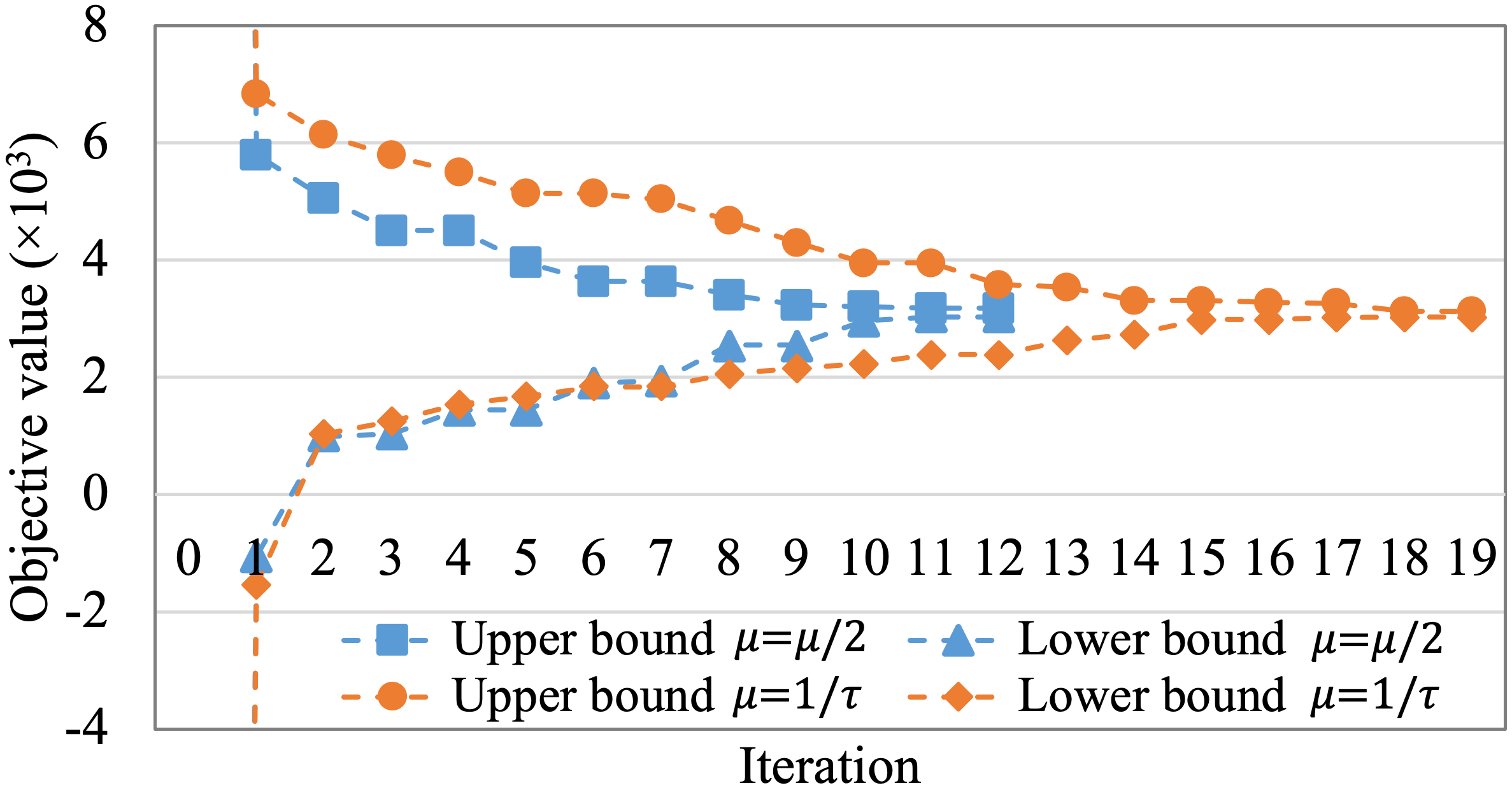}
		\vspace{-6mm}
		\caption{The convergence of bounds (\$) with different $\mu$ updates.}
		\label{fig:SBU_result2}
	\end{center}
\end{figure}
\vspace{-10mm}

\subsection{Real-world dataset}\label{subsec:case2}
\noindent The proposed methodology is applied to a real-world case study in Long Island, NY. The network includes 30 nodes, 340 links, and seven candidate charging facility locations in Stony Brook University campus, as shown in Figure \ref{fig:case2}(a).
The average arrivals over time-of-days are assumed to be 40, 50, 45, 30, 55 EVs for early AM, AM peak, mid-day, PM peak, and evening, respectively. EV users make their trips from three origins to three destinations in Long Island, NY. 
The origins and destinations are located in Brooklyn, NY, Queens, NY, Hempstead, NY, Hampton Bays, NY, Greenport, NY, and Mattituck, NY that represent an average trip length of 50 $miles$ from origins to campus. 
The real-world dataset includes 2,219,855 decision variables.
Due to the computational burden, a more efficient desktop computer with 72 2.10 GHz CPUs and 128 GB memory (compared to the test case) is utilized. 
The charging price is assumed to vary between a minimum of $l_i=0.1$ and maximum of $u_i=15.0, \forall i \in N$, per time step. We also set $\alpha=365$.

Figure \ref{fig:case2}(a) shows the location of candidate facilities to deploy EV chargers, selected facilities, and origins/destinations of EV trips. Since the majority of EV trips take the routes going to the south campus, three candidate facilities at nodes 14, 15 and 16 are suggested to satisfy the demand. Besides, the remaining demand will go to the center of the campus that encourages the deployment of chargers at nodes 18 and 19. 
As indicated in Figure \ref{fig:case2}(b), all EV users reach a charging facility or their final destination with an adequate charging level without violating the travel range limit of 50 $miles$.
\vspace{-2mm}
\begin{figure}[h]
  \begin{subfigure}
    \centering\includegraphics[scale=0.33]{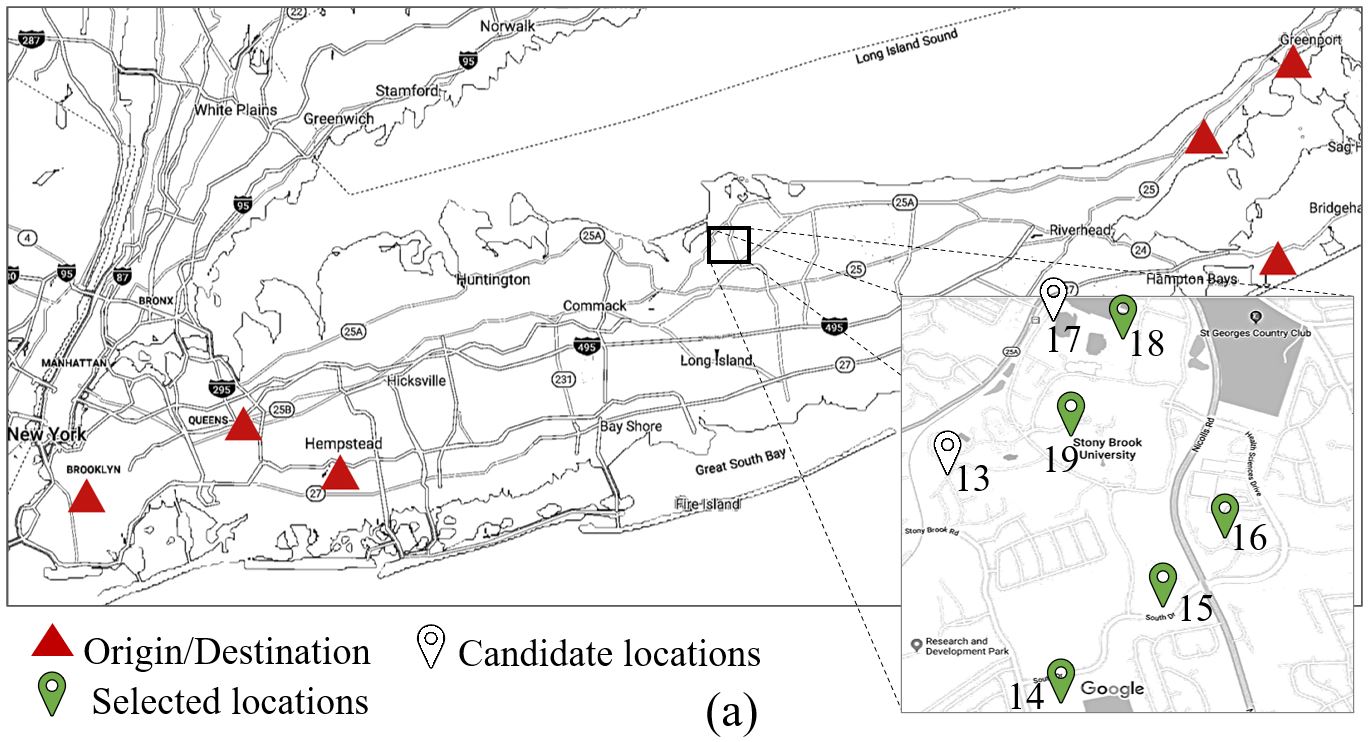}
  \end{subfigure}
  \begin{subfigure}
    \centering\includegraphics[scale=0.41]{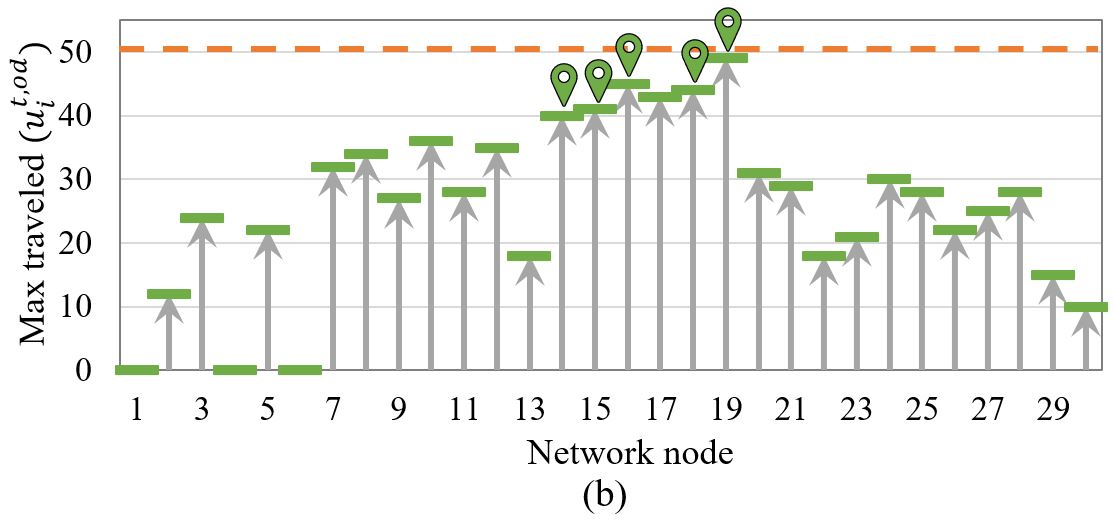}
    \vspace{-8mm}
	\caption{(a) Candidate facility locations in real-world case study. [Map source: Google, accessed August 7, 2019.] and (b) Max distance traveled from last visited node in med. demand.}  
	\label{fig:case2}
	\end{subfigure}
\end{figure}
%
%
Figure \ref{fig:chargingPriceReal} presents the impact of selected locations to install charging facilities on their pricing scheme over the planning horizon. The results indicate that higher demand in AM and PM peak hours imposes higher prices at the selected facilities.
When a higher number of chargers are available, the charging price decreases. This trend is observed at the beginning of the planning horizon (when the majority of chargers are available before new arrivals) and later during the day.
The increasing trend in prices towards the end of the day, in this case study, is due to the demand accumulation (as faculty/graduate students often tend to work for longer hours).
\vspace{-4mm}
\begin{figure}[H]
	\begin{center}
		\includegraphics[height=2.05in]{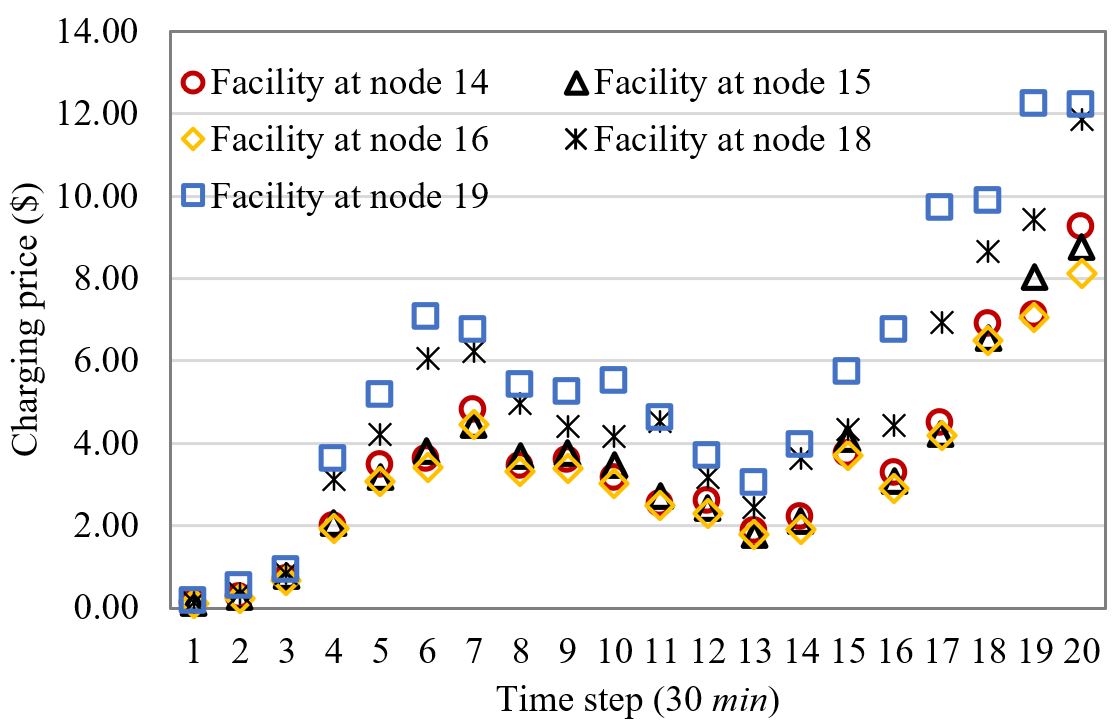}
		    \vspace{-7mm}
		\caption{Charging price (\$) at selected facilities over the planning horizon (8:00 AM to 5:30 PM).}
		\label{fig:chargingPriceReal}
	\end{center}
\end{figure}
\vspace{-12pt}
\noindent In addition, Figure \ref{fig:eachOccuPrice} suggest that a higher occupancy can impose higher charging prices in AM and PM peak hours. 
For instance, as shown in Figure \ref{fig:eachOccuPrice}(a), at charging facilities located farther away from popular locations (e.g., facility at node 14), higher prices can lead to lower occupancy (e.g., at time period 7) as users tend to park and charge at facilities closer to their destination. Another interpretation is that lower occupancy can lead to an increase in the charging price at certain un-popular locations to compensate for the associated loss of revenue, particularly at locations that do not get sufficiently high number of customers (except for users whose destinations are nearby). On the other hand, Figure \ref{fig:eachOccuPrice}(e) indicates that facility at node 19 increases the charging prices even with the growth in the occupancy (e.g., at time periods 6 and 7) that is primarily due to its relatively convenient location.

Table \ref{table:AverageReal} presents the average and standard deviation of charging prices over a day (8:00 AM to 5:30 PM) based on low and medium demand levels.
\vspace{-2mm}
\begin{table}[H]
	\caption{Average and Standard Deviation of Charging Prices at Selected Facilities for Low and Medium Demand Levels.}
	\vspace{-3mm}
	\begin{center}
		\small
		\begin{tabular}{M{1.0cm} M{1.7cm} M{1.7cm} M{1.7cm}}
			\hline
			\hline
			 & Low demand & Med demand & \% Diff \\[0.5ex]
			\hline			
			avg (\$) & 2.97 & 4.11 & 38.21 \\
			std (\$) & 1.78 & 2.58 & 45.19 \\
			\hline
		\end{tabular}
		\label{table:AverageReal}
	\end{center}
\end{table}
\vspace{-4mm}
\noindent Similar to the hypothetical case study, higher demands can increase the charging prices.
Besides, the selected charging facilities are fully utilized in the medium demand case. Given a large number of available chargers, EV users can have highly random choices that can lead to under-utilizing the un-popular charging locations (e.g., those that are farther away from main destinations). 
Such behavior can enforce the charging agency to adjust the charging prices by location based on the available demand level to better regulate the usage of all charging facilities. This behavior can be observed in Table \ref{table:AverageReal} as the average and standard deviation of charging prices are increased by 38.21\% and 45.19\% in the medium demand compared to the low demand level.

Figure \ref{fig:boundsReal} presents the convergence of upper and lower bounds of the proposed bi-level optimization program that reaches to a 2.89\% gap after 21 iterations in 68.8 $hr$ for the medium demand level. It can be observed that the solution considerably improves (i.e., the gap between the upper- and lower-bound substantially decreases) by adding cuts in the lower bounding procedure during the initial iterations. However, the improvement rate decreases as the lower bounding proceeds. 
The CPU time in the global optimization algorithm varies over iterations. A set of parametric upper bounds in initial iterations lead to lower computational time, while tightening the feasible area over iterations increases the complexity of lower bounding procedure (see the increasing trend in the runtime over the next iterations). 
%
\begin{figure}[H]
	\begin{center}
		\includegraphics[height=4.9 in]{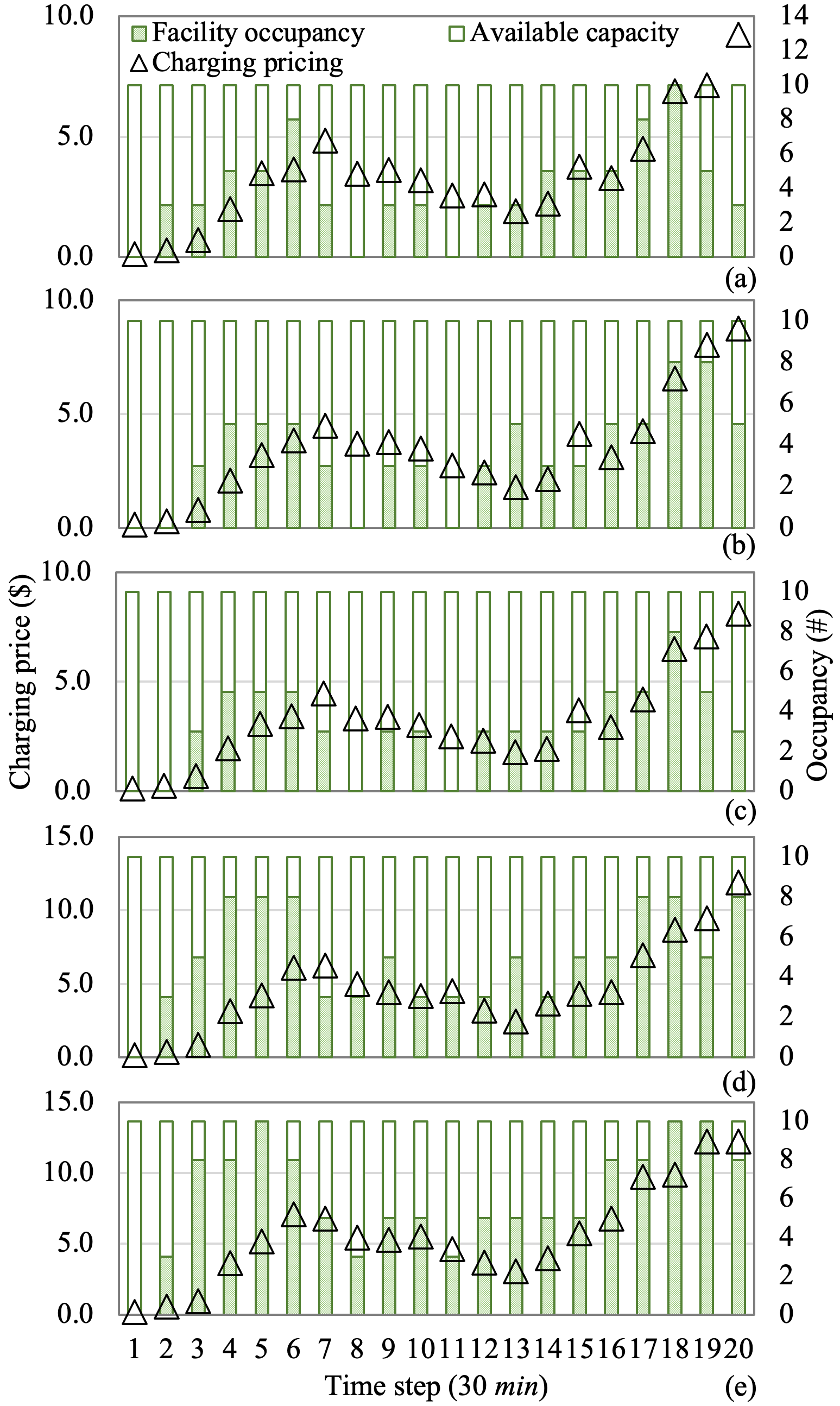}
		\vspace{-2mm}
		\caption{Charging and occupancy of facility over planning horizon at node (a) 14, (b) 15, (c) 16, (d) 18, and (e) 19.}
		\label{fig:eachOccuPrice}
	\end{center}
\end{figure}
%
%
%
%
\begin{figure}[H]
	\begin{center}
		\includegraphics[height=1.51 in]{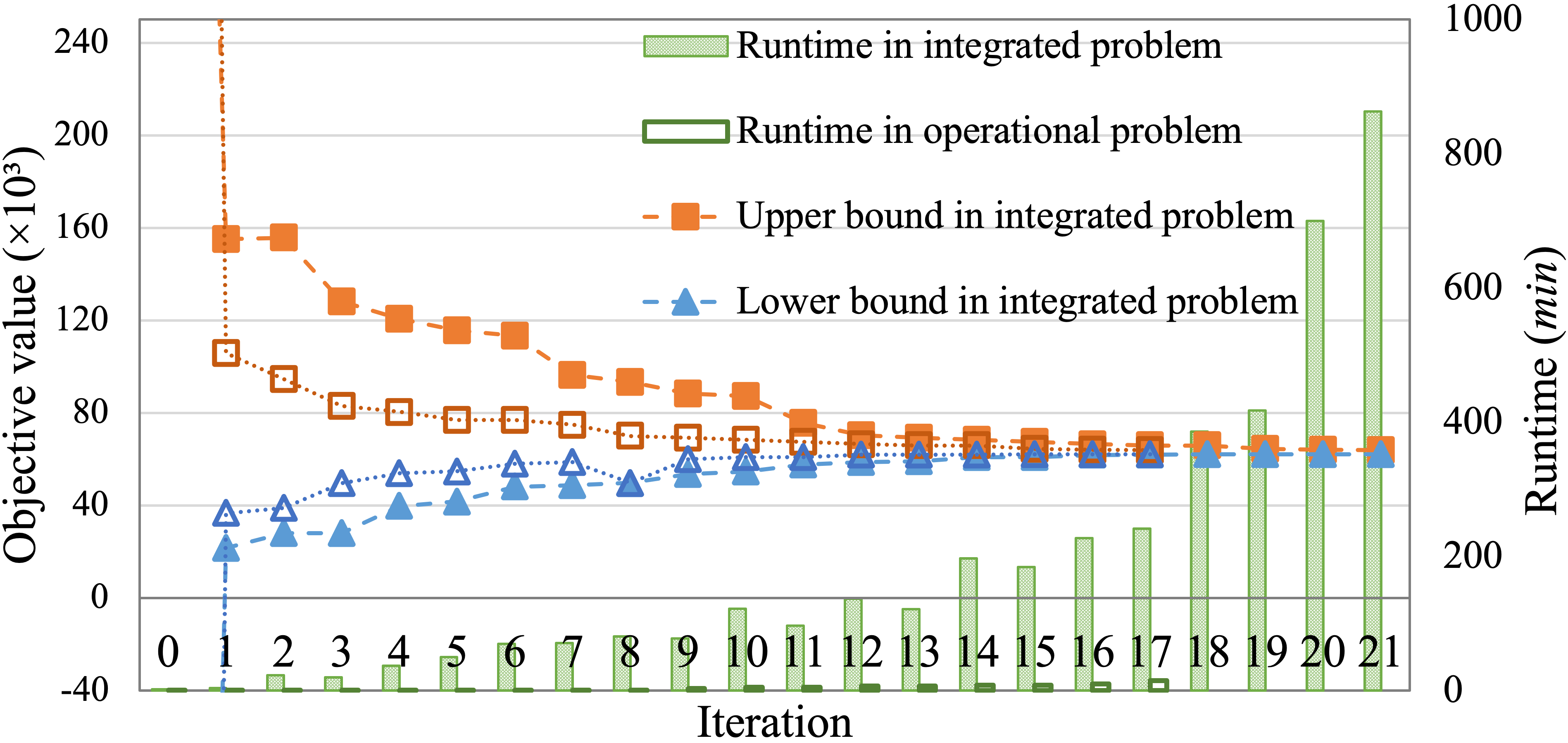}
		\vspace{-7mm}
		\caption{Convergence of upper bound and lower bound (\$) of the bi-level program and the runtime over iterations.}
		\label{fig:boundsReal}
	\end{center}
\end{figure}
\vspace{-4mm}
%
To evaluate the performance of the proposed solution technique on the operational aspect of the problem (i.e., dynamic pricing scheme and EV users’ decisions given the strategic charging network design), the proposed bi-level optimization program \eqref{eqref:UpperLevelObj}-\eqref{eqref:LLNonNeg1} is solved again given the optimal value of integer variables $\boldsymbol{y}$, $\boldsymbol{\eta}$, and $\boldsymbol{e}$ obtained from our original integrated optimization framework. Figure \ref{fig:boundsReal} shows the comparison of the test case results to those of the original framework. It can be observed that the computational time is significantly lower in the operational problem. 
As expected, the optimal operational decisions $\boldsymbol{x}$, $\boldsymbol{z}$, $\boldsymbol{v}$, and $\boldsymbol{p}$ are found in 1.2 $hr$ after 17 iterations.

Figure \ref{fig:objValueChargers} shows the values of upper- and lower-level objective functions over the algorithm iterations. The optimal number of chargers are also shown over iterations. 
Higher charging prices in the initial iterations (obtained from upper-level problem) has decreased the utilization. The initial iterations indicate that the algorithm suggests the deployment of a lower number of chargers (i.e., at most 5 chargers at each facility) to meet the demand. This trend is obtained without the consideration of EV users' travel time (obtained from lower-level problem). 
\vspace{-3mm}
\begin{figure}[H]
	\begin{center}
		\includegraphics[height=2.4 in]{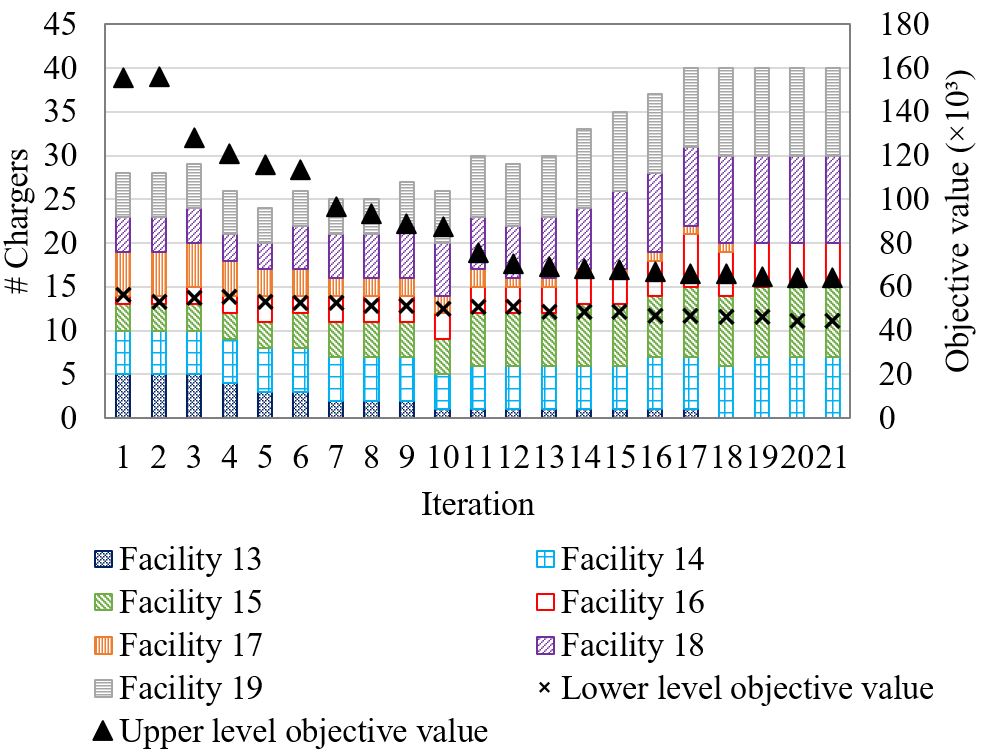}
		\vspace{-5mm}
		\caption{Upper-level and lower-level objective values (\$) and optimal number of chargers over iterations.}
		\label{fig:objValueChargers}
	\end{center}
\end{figure}
%
\noindent Therefore, the lower-level objective function will be added by improving the parametric upper bound (that is the parametric optimal value of lower-level problem as a function of upper-level variables) over iterations to capture the impact of EV travel times (see constraints \eqref{eq:relaxedCons}) and update the adequate number of chargers. 
Hence, a better understanding of the number of chargers will be obtained while the algorithm proceeds. As indicated, the optimal number of chargers is identified from iteration 17; however, their location has not fully optimized yet. 
Toward the convergence of the upper- and lower-bound (i.e., starting from iteration 19), the optimal location and capacity of charging facilities will be obtained that are facilities located at node 14 with 7 chargers, node 15 with 8 chargers, node 16 with 5 chargers, and nodes 18 and 19 each with 10 chargers.
\section{CONCLUSIONS}\label{sec:conclusion}
\noindent This study proposes a bi-level optimization program for network design and utilization management of EV charging facilities. The upper level aims to minimize facility deployment costs and maximize the revenue generated from charging activities, while the lower level aims to minimize the users' travel times and charging expenses. 
An iterative algorithm is applied to solve the problem that generates theoretical lower and upper bounds to the proposed bi-level problem. We find the lower bound by solving a global optimization problem that consists of the upper-level objective function, constraints of both upper- and lower-level models, and the parametric upper bound of the optimal solution to the lower-level problem. Then, the parametric upper bound is improved through an iterative process by dividing the feasible region of upper-level program into tighter bounds and finding an optimal solution at each interval. 
The lower bounding procedure iteratively adds cuts to improve the solution, while it decreases the efficiency of the algorithm due to the generation of new constraints \eqref{eq:relaxedCons}. The computational results show that the algorithm yields the desired optimality gap (i.e., 2.89\%) within 21 iterations in 68.8 hours.
The proposed solution technique is applied to a (i) hypothetical network to confirm its solution quality and (ii) real-world case study in Long Island, NY to evaluate its computational performance. 
The results present the impact of added cuts on the solution quality in each iteration. As observed, the optimal location of charging facilities that offer feasible paths for EV users are determined over the initial iterations, while the optimal number of charging spots are obtained over the final iterations.
An analysis of sensitivity has also been conducted to assess the behavior of the methodology based on the choice of updating processes and parameters.
It will be interesting, in the future, to compare the computational performance of the proposed methodology with a benchmark technique. 
It will be worthwhile to study the interaction, partnership, and competition among multiple charging network operators and assess the impact of their business on the pricing scheme and utilization of charging facilities using a multi-leader multi-follower model.
In addition, the impact of non-EVs on travel behavior and traffic assignment can be considered as an interesting future research.

\renewcommand{\baselinestretch}{1.0}
\vspace{-4mm}
\section*{Acknowledgment}
\vspace{-1mm}
The authors would like to thank Avery Mack of the State University of New York at Stony Brook, for her significant helps with data collection and preparation.
\vspace{-2mm}
\small
\bibliographystyle{IEEEtran}
\bibliography{EV}

\vspace{-1mm}
\appendices
\section{Proof of Proposition 1} \label{sec:appProof}
\begin{proof}
	We let $\bar{F}^{u} = F^{u*} + \bar{\varepsilon}_{F}^{u}$. The lower bounding procedure globally generates points $\bar{z} \in U^{\psi}(F^{u*} + \tilde{\varepsilon})$. Assuming $\tilde{\varepsilon} < \bar{\varepsilon}_{F}^{u}$, we have $\bar{z} \in U^{\psi}(\bar{F}^{u})$. The third assumption in Section \ref{ExactMethod} guarantees that sets $U^{k}$ and points $(\boldsymbol{x}^{k},\boldsymbol{z}^{k},\boldsymbol{v}^{k})$ can build the parametric upper bounds of optimal solution to the lower-level problem.
	Besides, adding cuts associated with newly defined sets and points in iteration $k$ leads to tightening the lower bound.
	We let $(\hat{\boldsymbol{y}},\hat{\boldsymbol{\eta}},\hat{\boldsymbol{p}},\hat{\boldsymbol{x}},\hat{\boldsymbol{z}},\hat{\boldsymbol{v}})$ refer to the new solution obtained from the lower bounding procedure, based on newly added cuts.
	Solution $(\hat{\boldsymbol{y}},\hat{\boldsymbol{\eta}},\hat{\boldsymbol{p}},\hat{\boldsymbol{x}},\hat{\boldsymbol{z}},\hat{\boldsymbol{v}})$ will be $\varepsilon$-optimal if $\hat{\boldsymbol{y}} = \bar{\boldsymbol{y}},\hat{\boldsymbol{\eta}} = \bar{\boldsymbol{\eta}}$ and $\hat{\boldsymbol{p}}$ is sufficiently close to $\bar{\boldsymbol{p}}$.
	Besides, the feasible region of lower-level problem does not depend on continuous upper-level decision variable $\boldsymbol{p}$. 
	Hence, according to Lemma 3 in \citep{mitsos2010global}, there exists a pair $(\bar{\boldsymbol{y}}, \bar{\boldsymbol{\eta}})$ such that point $(\boldsymbol{x}^{k},\boldsymbol{z}^{k},\boldsymbol{v}^{k})$ generated in \eqref{eq:secondStepObj}-\eqref{eq:secondStepCons4} satisfies
	\vspace{-1mm}
	\begin{align}
		& 
		\eqref{eqref:LLCons2d}-\eqref{eq:BPR},\,\eqref{eqref:LLbinary}-\eqref{eqref:rangeAnxiety6},\,\eqref{eqref:LLCons2c}-\eqref{eqref:LLNonNeg1},\,\mbox{and}\,\nonumber\\
		&
		\sum_{a \in A} \left(\int_{0}^{v_{a}^{t,k}} R_{a}^{t}(\omega)\,d\omega+ \gamma\,\sum_{i \in N}\sum_{od \in OD} \bar{p}_{i}^{t}\,x_{a}^{t,od,i,k}\right) \nonumber\\
		&\epc \epc \epc \epc \epc \epc \epc \epc \epc \le {F}^{l*}(\bar{\boldsymbol{y}},\bar{\boldsymbol{\eta}},\bar{\boldsymbol{p}}) + \varepsilon_{F2}^{l}. \label{eq:proof1}
		\vspace{-1mm}
	\end{align}
	%
    %
	Based on assumption $\varepsilon_{F}^{l}-\tilde{\varepsilon}-\varepsilon_{F2}^{l} > 0$, there exists $\delta_{1} > 0$: 
	\begin{align}
	 &{F}^{l*}(\bar{\boldsymbol{y}},\bar{\boldsymbol{\eta}},\bar{\boldsymbol{p}}) \le {F}^{l*}(\bar{\boldsymbol{y}},\bar{\boldsymbol{\eta}},{\boldsymbol{p}})\nonumber\\
	 &+ .5(\varepsilon_{F}^{l}-\tilde{\varepsilon}-\varepsilon_{F2}^{l}), \forall \boldsymbol{p}:
	||\boldsymbol{p} - \bar{\boldsymbol{p}}|| < \delta_{1}. \label{eq:proof2}
	\end{align}
	By the continuity of \\
	$\sum_{a \in A} \left(\int_{0}^{v_{a}^{t,k}} R_{a}^{t}(\omega)\,d\omega+ \gamma\,\sum_{i \in N}\sum_{od \in OD} \bar{p}_{i}^{t}\,x_{a}^{t,od,i,k}\right)$ on the host set of $\boldsymbol{p}$, we can derive the following for $\delta_{2} > 0$:
	\vspace{-1mm}
	\begin{align}
		&\sum_{a \in A} \left(\int_{0}^{v_{a}^{t,k}} R_{a}^{t}(\omega)\,d\omega+ \gamma\,\sum_{i \in N}\sum_{od \in OD} {p}_{i}^{t}\,x_{a}^{t,od,i,k}\right)\nonumber\\
		&\le \sum_{a \in A} \left(\int_{0}^{v_{a}^{t,k}} R_{a}^{t}(\omega)\,d\omega+ \gamma\,\sum_{i \in N}\sum_{od \in OD} \bar{p}_{i}^{t}\,x_{a}^{t,od,i,k}\right) \nonumber\\
		&
		+ .5(\varepsilon_{F}^{l}-\tilde{\varepsilon}-\varepsilon_{F2}^{l}),\, \forall \boldsymbol{p}:
		||\boldsymbol{p} - \bar{\boldsymbol{p}}|| < \delta_{2}. \label{eq:proof3}
	\end{align}
	\vspace{-2mm}
	According to \eqref{eq:proof1}-\eqref{eq:proof3}, we can conclude
	%
	\begin{align}
	&\sum_{a \in A} \left(\int_{0}^{v_{a}^{t,k}} R_{a}^{t}(\omega)\,d\omega+ \gamma\,\sum_{i \in N}\sum_{od \in OD} {p}_{i}^{t}\,x_{a}^{t,od,i,k}\right)\nonumber\\
	& \le {F}^{l*}(\bar{\boldsymbol{y}},\bar{\boldsymbol{\eta}},{\boldsymbol{p}})+ \varepsilon_{F}^{l}-\tilde{\varepsilon},
	\forall \boldsymbol{p}: ||\boldsymbol{p}-\bar{\boldsymbol{p}}|| < \text{min}\{\delta_{1}, \delta_{2}\}. \nonumber \label{eq:proof4}
	\end{align}
	%
	On the other hand, according to \eqref{eq:proof1}, points $(\boldsymbol{x}^{k},\boldsymbol{z}^{k},\boldsymbol{v}^{k})$ are $\varepsilon$-optimal to the lower-level problem for $\forall(\boldsymbol{y} = \bar{\boldsymbol{y}}, \boldsymbol{\eta} = \bar{\boldsymbol{\eta}}),\, \forall \boldsymbol{p}:
	||\boldsymbol{p} - \bar{\boldsymbol{p}}|| < \delta = \text{min}\{\delta_{1}, \delta_{2}\}$.
	Thus, after a finite horizon, if the lower bound is feasible, a point $(\hat{\boldsymbol{y}},\hat{\boldsymbol{\eta}},\hat{\boldsymbol{p}},\hat{\boldsymbol{x}},\hat{\boldsymbol{z}},\hat{\boldsymbol{v}})$, with $\hat{\boldsymbol{y}} = \bar{\boldsymbol{y}},\hat{\boldsymbol{\eta}} = \bar{\boldsymbol{\eta}}$, and $\boldsymbol{p}:
	||\boldsymbol{p} - \bar{\boldsymbol{p}}|| < \delta$ will be generated. Since this point is the solution to
	\eqref{eq:exactObj1}-\eqref{eq:relaxedCons}, the following inequality is valid:
    %
	$F^{l}(\hat{\boldsymbol{y}},\hat{\boldsymbol{\eta}},\hat{\boldsymbol{p}},\hat{\boldsymbol{x}},\hat{\boldsymbol{z}},\hat{\boldsymbol{v}}) \le F^{l*}(\hat{\boldsymbol{y}},\hat{\boldsymbol{\eta}},\hat{\boldsymbol{p}}) + \varepsilon_{F}^{l}-\tilde{\varepsilon}.$
	%
	
	Additionally, we have $\text{LBD} \ge F^{u}(\hat{\boldsymbol{y}},\hat{\boldsymbol{\eta}},\hat{\boldsymbol{p}},\hat{\boldsymbol{x}},\hat{\boldsymbol{z}},\hat{\boldsymbol{v}}) -\tilde{\varepsilon}$. By the feasibility of point $(\hat{\boldsymbol{y}},\hat{\boldsymbol{\eta}},\hat{\boldsymbol{p}},\hat{\boldsymbol{x}},\hat{\boldsymbol{z}},\hat{\boldsymbol{v}})$ in the bi-level program \eqref{eqref:ULDisutility1}-\eqref{eq:BPR},\,\eqref{eqref:LLbinary}-\eqref{eqref:rangeAnxiety6},\,\eqref{eqref:UpperLevelObj}-\eqref{eqref:LLNonNeg1},\,\mbox{and}\, \eqref{eqref:waitingApprox}, the upper bounding procedure 
	\eqref{eq:upperBoundObj}-\eqref{eq:upperBound8} generates an upper bound as $\text{UBD} \le F^{u}(\hat{\boldsymbol{y}},\hat{\boldsymbol{\eta}},\hat{\boldsymbol{p}},\hat{\boldsymbol{x}},\hat{\boldsymbol{z}},\hat{\boldsymbol{v}}) + \tilde{\varepsilon}$. Therefore, the assumption $\tilde{\varepsilon} < \varepsilon_{F}^{u}/2$ will lead to $\text{UBD} - \text{LBD} \le 2\tilde{\varepsilon} \le \varepsilon_{F}^{u}$, which confirms finitely termination of the global optimization algorithm. 
	\end{proof}

\section{Nomenclature} \label{sec:app1}
\vspace{-5mm}
\begin{table}[H] 
	\caption{Definitions of sets, decision variables and parameters.}
	\vspace{-5mm}
	\begin{center}
		\tiny
		\begin{tabular}{@{}ll}
			\hline\hline
			\textbf{Sets}\\[0.0ex]
			\hline
			$\Gamma$ & Set of all time steps; i.e., $\{0, 1, \dots, T-1\}$  \\
			$N$ & Set of nodes \\
			$A$ & Set of arcs\\
			$A_{i}^{-}$, $A_{i}^{+}$ & Set of inbound and outbound network arcs to/from node $i \in N$\\
			$OD$ & Set of origin-destination \\
			$\lambda_{od}^{t}$ & Set of users with $od \in OD$ at time $t \in \Gamma$ \\[1.0ex]
			
			\textbf{Decision variables}\\[0.0ex]
			\hline
			$y_{i}$ & Binary variable, 1 if there is a  facility at node $i \in N$, or 0 otherwise  \\
			$\eta_{i}$ & Number of chargers at node $i \in N$\\
			$p_{i}^t$ & Price of getting charge at node $i \in N$ at time $t \in \Gamma$\\
			$\sigma_{od}^{t}$ & Equilibrium disutility identical for users on $od \in OD$ at $t \in \Gamma$\\
			$x_{a}^{t,od,i}$ & Flow of arc $a\in A$ on $od\in OD$ to charge at node $i\in N$ at time $t\in \Gamma$\\
			$z_{a}^{t,od}$ & Flow of arc $a \in A$ on $od \in OD$ at time $t \in \Gamma$\\
			$v_{a}^{t}$ & Aggregated flow of arc $a \in A$ at time $t \in \Gamma$\\
			$R_{a}^{t}(v_{a}^{t})$ & Travel time of arc $a \in A$ at time $t \in \Gamma$\\
			$e_{a}^{t,od}$ & Binary variable, 1 if arc $a \in A$ is on a feasible path for EV users\\
			&on $od \in OD$ at time $t \in \Gamma$, or 0 otherwise\\
			$u_{i}^{od}$ & Maximum distance traveled from the last visited charging facility located\\
			&at node $i \in N$ with $od \in OD$\\
			$u_{i}^{'od}$ & Dummy variable, 0 at charging facility located at $i \in N$\\
			$U$ & Host set of all upper-level variables $(\boldsymbol{y},\boldsymbol{\eta},\boldsymbol{p})$\\
			$L$ & Host set of all lower-level variables $(\boldsymbol{x},\boldsymbol{z},\boldsymbol{v})$\\[1.0ex]
		
		\textbf{Parameters}\\[0.0ex]
			\hline
			$C_{i}$ & Unit cost of building charging facility at node $i \in N$ \\
			$f_{i}^{t}$ & The occupancy of charging facility at node $i \in N$ at time $t \in \Gamma$\\
			$\hat{\eta}_i^t$ & Available capacity of charging facility at node $i \in N$ at time $t \in \Gamma$\\
			$M$ & Large positive constant \\
			$B$ & Available budget \\
			$l_{i}$ & Lower bound for charging price at node $i \in N$\\
			$u_{i}$ & Upper bound for charging price at node $i \in N$\\
			$\Delta$ & Driving range of the EVs \\
			$g_{od}^{t}$ & Linear demand curve's intercept for EV users with $od \in OD$ at time $t \in \Gamma$ \\
			$b$ & The elasticity coefficient of the demand function \\
			$\eta_{max}$ & Maximum physical capacity \\
			$\theta$ & Average charging duration \\
			$\xi_{i}^{t}$ & EV users' arrival rate in node $i \in N$ at time $t \in \Gamma$\\
			$\delta_a$ & Length of arc $a \in A$\\
			$q,w$ & BPR function parameters\\
			$\nu$ & Time window for finding an available charger\\
			$c_a$ & Traffic capacity of arc $a \in A$\\
			$\kappa$ & Lower bound for the probability of finding an available charger\\
			$\Omega_i$ & The slope of the linear waiting function at node $i \in N$\\
			$\mu$ & The update factor in subroutine\\
			\hline
		\end{tabular}
		\label{table:definitions}
	\end{center}
\end{table}

\end{document}